\documentclass{amsart}
\usepackage{amsmath}
\usepackage{amsthm,amscd,amssymb,amsfonts, amsbsy}
\usepackage{latexsym}
\usepackage{txfonts, mathrsfs}

\usepackage{color}
\usepackage{enumerate}

\numberwithin{equation}{section}

\theoremstyle{plain}
\newtheorem{theorem}[equation]{Theorem}
\newtheorem{lemma}[equation]{Lemma}
\newtheorem{corollary}[equation]{Corollary}

\theoremstyle{definition}

\newtheorem{remark}[equation]{Remark}

\newtheorem*{acknowledgment}{Acknowledgment}
\newtheorem*{CLB}{Condition (LB)}
\newtheorem*{CIH}{Condition (IH)}
\newtheorem*{CLH}{Condition (LH)}

\theoremstyle{remark}

\newtheorem*{case1}{Case 1: $\abs{x-y} \leq \sqrt{t-s}<R_{max}$}
\newtheorem*{case2}{Case 2: $\sqrt{t-s}< \abs{x-y} \wedge R_{max}$}
\newtheorem*{case3}{Case 3: $R_{max} \leq \sqrt{t-s}$}

\newcommand{\supp}{\operatorname{supp}}
\newcommand{\dist}{\operatorname{dist}}

\newcommand{\osc}{\operatorname*{osc}}
\newcommand{\esssup}{\operatorname*{ess\,sup}}

\newcommand{\mysection}[1]{\section{#1}
\setcounter{equation}{0}}

\renewcommand{\vec}[1]{\boldsymbol{#1}}

\newcommand{\bR}{\mathbb R}

\newcommand\rV{\mathring{V}}

\newcommand\sL{\mathscr{L}}
\newcommand\sP{\mathscr{P}}
\newcommand\sLt{{}^t\!\mathscr{L}}

\newcommand{\LB}{\mathrm{(LB)}}
\newcommand{\IH}{\mathrm{(IH)}}
\newcommand{\PH}{\mathrm{(PH)}}
\newcommand{\LH}{\mathrm{(LH)}}
\newcommand{\VMO}{\mathrm{VMO}}

\newcommand{\ip}[1]{\left\langle#1\right\rangle}

\providecommand{\set}[1]{\{#1\}}
\providecommand{\Set}[1]{\left\{#1\right\}}
\providecommand{\bigset}[1]{\bigl\{#1\bigr\}}

\providecommand{\abs}[1]{\lvert#1\rvert}
\providecommand{\Abs}[1]{\left\lvert#1\right\rvert}
\providecommand{\bigabs}[1]{\bigl\lvert#1\bigr\rvert}
\providecommand{\Bigabs}[1]{\Bigl\lvert#1\Bigr\rvert}

\providecommand{\norm}[1]{\lVert#1\rVert}

\renewcommand{\epsilon}{\varepsilon}
\renewcommand{\qedsymbol}{$\blacksquare$}

\begin{document}
\title[Green's functions in time-varying domains]{Green's functions for parabolic systems of second order in time-varying domains}

\author[H. Dong]{Hongjie Dong}
\address[H. Dong]{Division of Applied Mathematics, Brown University, 182 George Street, Box F, Providence, RI 02912, USA}
\email{Hongjie\_Dong@brown.edu}

\author[S. Kim]{Seick Kim}
\address[S. Kim]{Department of Mathematics, Yonsei University, Seoul 120-749, Republic of Korea}
\curraddr{Department of Computational Science and Engineering, Yonsei University, Seoul 120-749, Republic of Korea}
\email{kimseick@yonsei.ac.kr}


\begin{abstract}
We construct Green's functions for divergence form, second order parabolic systems in non-smooth time-varying domains whose boundaries are locally represented as graph of functions that are Lipschitz continuous in the spatial variables and $1/2$-H\"older continuous in the time variable, under the assumption that weak solutions of the system satisfy an interior H\"older continuity estimate.
We also derive global pointwise estimates for Green's function in such time-varying domains under the assumption that weak solutions of the system vanishing on a portion of the boundary satisfy a certain local boundedness estimate and a local H\"older continuity estimate.
In particular, our results apply to complex perturbations of a single real equation.
\end{abstract}

\maketitle

\mysection{Introduction} \label{intro}
 
Green's functions play an important role in the solution of elliptic and parabolic partial differential equations.
There is a large literature on Green's functions of uniformly elliptic and parabolic equations in divergence form.
Green's functions of elliptic equations of divergence form with $L_\infty$ coefficients have been extensively studied by Littman et al. \cite{LSW} and Gr\"uter and Widman \cite{GW};  see also \cite{DM, Fuchs84, Fuchs86}.
Recently, Hofmann and Kim \cite{HK07} gave a unified approach in studying Green's functions for both scalar equations and systems of elliptic type; see also \cite{DK09}.
For parabolic equations, Aronson \cite{Aronson} established two-sided Gaussian bounds for the fundamental solutions of parabolic equations in divergence form with $L_\infty$ coefficients; see also \cite{Auscher, Davies, FS, HK04, Kim, PE} and references therein for related results.
In a recent paper by Cho and the authors \cite{CDK}, we proved that if weak solutions of a given parabolic system satisfy an interior H\"older continuity estimate, then the Green's function of the system exists in any cylindrical domain.
In the scalar case, such an interior H\"older continuity estimate is a consequence of Nash \cite{Nash} and Moser \cite{Moser}, and also such an estimate is available for weak solutions of a system if, for example, its principal coefficients are VMO in the spatial variables.
However, the construction of Green's function in \cite{CDK} heavily relied on the results by Ladyzhenskaya  and Ural'tseva that are available only for cylindrical domains.
In another recent article by Cho and the authors \cite{CDK2}, we demonstrated how one can derive global pointwise estimates for the Green's function in a cylindrical domain by using a local boundedness estimate and a local H\"older estimate for the weak solutions of the parabolic system vanishing on a portion of the boundary.

The aim of this article is to give results similar to those of \cite{CDK, CDK2} for a class non-smooth time-varying domains whose boundaries are given locally as graph of functions that are Lipschitz continuous in the spatial variables and $1/2$-H\"older continuous in the time variable, which hereafter shall be referred to as time-varying $H_1$ domains.
There are many papers dealing with parabolic equations in this type of time-varying domains.
Lewis and Murray \cite{LM95} considered a domain in $\bR\times \bR^d$ of the form $\set {(t,x)\colon\, x^d>f(t,x')}$, where $x'=(x^1,\cdots,x^{d-1})$ and $f$ is a function that is Lipschitz in $x'$ and whose $t$-derivative of order $1/2$ belongs to BMO, which is slightly stronger than $f \in H_1$.
The non-cylindrical domains considered by Hofmann and Lewis in their important paper \cite{HL96} on $L_2$ boundary value problems for the heat equation are also included in time-varying $H_1$ domains; see also \cite{HN02, Ny08, RN}.
Brown et al. \cite{BHL} investigated weak solutions of parabolic equations in time-varying $H_1$ domains and proved the unique solvability of Dirichlet boundary value problems.
We shall in fact utilize their result in constructing the Green's function in time-varying $H_1$ domains.
In contrast, there is little literature on Green's functions of parabolic equations in non-cylindrical domains and to the best of our knowledge, there is no literature dealing with Green's function for parabolic systems of second order with $L_\infty$ coefficients in time-varying $H_1$ domains; see the remarks made in the last paragraph of the introduction.

We denote by $\vec u=(u^1,\ldots, u^m)^T$ a vector-valued function of $d+1$ independent variables $(t,x^1,\ldots,x^d)=(t,x)=X$.
We consider parabolic systems of  second-order
\begin{equation}					\label{eq0.1}
\sL \vec u := \vec u_t - D_\alpha (\vec A^{\alpha\beta} D_\beta \vec u),
\end{equation}
where the usual summation conventions are assumed and $\vec{A}^{\alpha\beta}=\vec{A}^{\alpha\beta}(X)$, for $\alpha,\beta=1,\ldots,d$, are $m\times m$ matrices whose entries are $L_\infty$ functions satisfying the strong ellipticity condition; see Section~\ref{sec:sps} for the details.
We emphasize that the coefficients are not assumed to be time independent or symmetric.
We will later impose some further assumptions on the operator $\sL$ but not explicitly on its coefficients.
By a Green's function for the system \eqref{eq0.1} in a time-varying $H_1$ domain $\Omega$ we mean an $m\times m$ matrix valued function $\vec G(X,Y)=\vec G(t,x,s,y)$ which satisfies the following for all $Y\in\Omega$: 
\begin{align*}
\sL \vec G(\cdot,Y)=\delta_Y(\cdot) I_m\;\text{ in }\;\Omega,\\
\vec G(\cdot,Y)=0\;\text{ on }\; \sP\Omega,
\end{align*}
where $\delta_{Y}(\cdot)$ is a Dirac delta function, $I_m$ is $m\times m$ identity matrix, and $\sP\Omega$ denotes the parabolic boundary of $\Omega$; see Section~\ref{sec:dgf} for more precise definition.
In this article, we prove that if weak solutions of \eqref{eq0.1} satisfy an interior H\"older continuity estimate, then there exists a unique Green's function in $\Omega$ and it satisfies some natural growth properties; see Theorem~\ref{thm1} below.
Moreover, we show that the Green's function also satisfies the following familiar property:
\[
\lim_{t\to s_+} \vec G(t,x,s,\cdot)=\delta_{x}(\cdot) I_m\quad\text{on }\;\omega(s)=\set{y\in\bR^d\colon\, (s,y)\in\Omega}.
\]
We also derive the following global Gaussian estimate for the Green's function in a time-varying $H_1$ domain by using a local boundedness estimate for the weak solutions of \eqref{eq0.1} vanishing on a portion of the boundary:
For any $T>0$, there exists $N>0$ such that for all $X=(t,x)$ and $Y=(s,y) $ in $\Omega$ satisfying $0<t-s<T$, we have
\begin{equation}        \label{eq0.2}
\abs{\vec G (t,x,s,y)} \leq \frac{N}{(t-s)^{d/2}}\exp\Set{-\frac{\kappa\abs{x-y}^2}{t-s}},
\end{equation}
where $\kappa>0$ is a constant independent of $T$; see Theorem~\ref{thm2} and Remark~\ref{rmk3.15}.
In particular, the above estimate \eqref{eq0.2} holds in the scalar case (i.e., when $m=1$) and also in the case of $L_\infty$-perturbation of diagonal systems; see Corollary~\ref{cor1} and Section~\ref{sec:AD} below.
In fact, in such cases, a stronger estimate is available near the boundary.
For any $T>0$, there exists $N>0$ such that for all $X=(t,x)$ and $Y=(s,y) $ in $\Omega$ satisfying $0<t-s<T$, we have
\begin{equation}							\label{eq0.3}
\abs{\vec G(t,x,s,y)} \leq N \left(1 \wedge \frac {d(X)\;} {\abs{X-Y}_\sP} \right)^{\mu} \left(1 \wedge \frac{d(Y)\;} {\abs{X-Y}_\sP}\right)^{\mu} \frac{1}{(t-s)^{d/2}}\exp\Set{-\frac{\kappa\abs{x-y}^2}{t-s}},
\end{equation}
where $\kappa>0$ and $\mu\in (0,1]$ are constants independent of $T$,  and we used the notation $a\wedge b=\min(a,b)$, $\abs{X-Y}_\sP=\max(\sqrt{\abs{t-s}},\abs{x-y})$, and $d(X)=\inf \set{\abs{Z-X}_{\sP}\colon\, Z\in \partial \Omega}$.
We show how to derive a global estimate like \eqref{eq0.3} for the Green's function in a time-varying $H_1$ domain by using a local H\"older continuity estimate for the weak solutions of \eqref{eq0.1} vanishing on a portion of the boundary; see Theorem~\ref{thm3} and Remark~\ref{rmk3.19}.
As mentioned above, the estimate \eqref{eq0.3} particularly holds in the case of $L_\infty$-perturbation of diagonal system as well as in the scalar case; see Corollary~\ref{cor3} and Section~\ref{sec:AD}.

The organization of the paper is as follows.
In Section~\ref{sec:nd}, we introduce some notation and definitions including the precise definitions of time-varying $H_1$ domains and Green's functions of the system \eqref{eq0.1} in such domains.
In Section~\ref{main}, we state our main theorems  and give a few remarks concerning extensions of them.
In Section~\ref{sec:app}, we present some applications of our main results including applications to the scalar case, $L_\infty$-perturbation of diagonal systems, and systems with $\VMO_x$ coefficients.
We provide proofs of our main theorems in Section~\ref{sec:p} and some technical lemmas are proved in the appendix.

Finally, several remarks are in order.
In the scalar case, there are a few papers discussing Green's functions in non-cylindrical domains.
However, we believe that even in the scalar case, our results give still new perspectives on Green's functions.
In \cite{Ny97}, Nystr\"om constructed Green's functions in bounded time-varying $H_1$ domains utilizing the fundamental solutions and the caloric measures, and in doing so, he made a qualitative assumption that the coefficients are smooth in order to have well-defined concept of solutions; i.e. to assume that all solutions are classical ones.
The main drawback of this kind of approach is that it is not well suited to handle unbounded domains, especially domains with unbounded cross-sections such as the graph domains considered by Hofmann and Lewis \cite{HL96}.
The novelty of our paper lies in presenting a powerful unifying method that establishes the existence and various estimates for the Green's function of parabolic equations and systems with $L_\infty$ coefficients in time-varying $H_1$ domains including the graph domains.
Also, even though we impose some conditions on the operator $\sL$ in the vectorial case, we do not make any smoothness assumption on its coefficients in order to assume that the solutions of the system are classical.
Moreover, the treatment of $L_\infty$-perturbation of diagonal systems is a unique feature of our paper and we believe that it could find some interesting applications in the complex perturbation theory for the Dirichlet problem of second order parabolic equations in time-varying domains.

\mysection{Notation and Definitions} 					\label{sec:nd}
\subsection{Basic notation}
We mostly follow notation employed in Ladyzhenskaya et al. \cite{LSU}, supplemented by that used in Lieberman \cite{Lieberman}.
First we use $X=(t,x)=(t,x^1,\ldots,x^d)$ to denote a point in $\bR^{d+1}$ with $d\geq 1$ and we denote $X'=(t,x')=(t,x^1,\ldots, x^{d-1})\in \bR^d$ so that $X=(X',x^d)$. We also write $Y=(s,y)=(s,y',y^d)=(Y',y^d)$.
We denote
\[
a\wedge b=\min(a,b),\quad a\vee b=\max(a,b)\quad\text{for }\; a,b\in[-\infty, \infty].
\]
We define the parabolic distance in $\bR^{d+1}$ and $\bR^d$, respectively, by
\[
\abs{X-Y}_{\sP}=\sqrt{\abs{t-s}} \vee \abs{x-y},\quad \abs{X'-Y'}_{\sP}=\sqrt{\abs{t-s}} \vee \abs{x'-y'},
\]
where $\abs{\,\cdot\,}$ denotes the usual Euclidean norm, and write $\abs{X}_{\sP}=\abs{X-0}_{\sP}$.
We define the parabolic H\"older norm as follows:
\[
\abs{u}_{\mu/2,\mu;\Omega}=[u]_{\mu/2, \mu;Q}+\abs{u}_{0;\Omega}: = \sup_{\substack{X, Y \in \Omega\\ X\neq Y}} \frac{\abs{u(X)-u(Y)}}{\abs{X-Y}_\sP^ \mu} + \sup_{X\in \Omega}\,\abs{u(X)},\quad \mu\in(0,1].
\]
By $C^{\mu/2,\mu}(\Omega)$ we denote the set of all bounded measurable functions $u$ on $\Omega$ for which $\abs{u}_{\mu/2,\mu;\Omega}$ is finite.
We write $D_i u=D_{x^i} u=\partial u/\partial x^i$ and $u_t=\partial u /\partial t$.
We also write $Du=D_x u$ for the vector $(D_1 u,\ldots, D_d u)$.
We use the following notation for basic cylinders in $\bR^{d+1}$:
\begin{gather*}
Q(X_0,R)=\set{X\in \bR^{d+1}\colon\,  \abs{X-X_0}_{\sP}<R},\\
Q_{-}(X_0,R)=\set{X=(t,x)\in \bR^{d+1}\colon\, \abs{X-X_0}_{\sP}<R,\, t<t_0},\\
Q_{+}(X_0,R)=\set{X=(t,x)\in \bR^{d+1}\colon\, \abs{X-X_0}_{\sP}<R,\, t>t_0}.
\end{gather*}
We also use the ball $B(x_0,r)=\set{x\in \bR^d \colon\, \abs{x-x_0}<r}$.
For convenience, the parameter $X_0$ (or $x_0$) in the notation above is omitted if $X_0=0$ (or $x_0=0$, respectively).
We use $\Omega$ to denote a domain (open connected set) in $\bR^{d+1}$.
For a fixed number $t_0$, we write $\omega(t_0)$ for the set of all points $(t_0,x)$ in $\Omega$, and write $I(\Omega)$ for the set of all $t$ such that $\omega(t)$ is not empty.
For $-\infty \leq t_0 < t_1 \leq \infty$, we denote
\[
\Omega(t_0,t_1)=\set{X=(t,x)\in \Omega\colon\, t_0<t<t_1}.
\]
The parabolic boundary $\sP \Omega$ is defined to be the set of all points $X_0\in \partial\Omega$ such that for any $\epsilon>0$, the cylinder $Q_{-}(X_0, \epsilon)$ contains points not in $\Omega$.
We define $B\Omega$ to be the set of all points $X_0\in \sP \Omega$ such that there is a positive $R$ with $Q_{+}(X_0,R)\subset \Omega$ and $S\Omega= \sP\Omega \setminus {B\Omega}$.
We define the ``time-reversed" parabolic boundary $\widetilde\sP \Omega$ to be the set of all points $X_0\in \partial\Omega$ such that for any $\epsilon>0$, the cylinder $Q_{+}(X_0, \epsilon)$ contains points not in $\Omega$.
We also define
\[
\Omega[X,R] =  \Omega\cap Q(X,R),\quad \sP\Omega[X,R]=\sP\Omega \cap Q(X,R),\quad \widetilde \sP\Omega[X,R]=\widetilde \sP\Omega \cap Q(X,R),
\]
and similarly $\Omega_{\pm}[X,R]$, $\sP\Omega_{\pm}[X,R]$, and $\widetilde \sP \Omega_{\pm}[X,R]$.
Finally, we define distance functions
\begin{gather*}
d_\Omega(X)=d(X)=\inf \set{\abs{Y-X}_{\sP}\colon\, Y\in \partial \Omega},\\
d_\Omega^{-}(X)=d^{-}(X)=\inf \set{\abs{Y-X}_{\sP}\colon\, Y\in \sP \Omega,\, s \leq t},\\
d_\Omega^{+}(X)=d^{+}(X)=\inf \set{\abs{Y-X}_{\sP}\colon\, Y\in \widetilde \sP \Omega,\, s \geq t}.
\end{gather*}

\subsection{Strongly parabolic systems}					\label{sec:sps}
Let the operator $\sL$ be defined as in \eqref{eq0.1}.
We assume that the coefficient of $\sL$ are defined in the whole space $\bR^{d+1}$ in a measurable way and that the principal coefficients $\vec{A}^{\alpha\beta}$ with the components $A^{\alpha\beta}_{ij}$ satisfy the strong ellipticity
\begin{equation}    \label{eqP-02}
\sum_{i,j=1}^m \sum_{\alpha,\beta=1}^d A^{\alpha\beta}_{ij}(X)\xi^j_\beta \xi^i_\alpha \geq \nu\sum_{i=1}^m \sum_{\alpha=1}^d \bigabs{\xi^i_\alpha}^2=:\nu \bigabs{\vec \xi}^2, \quad\forall \vec\xi \in \bR^{dm},\quad\forall X\in\bR^{d+1},
\end{equation}
and the uniform boundedness condition
\begin{equation}    \label{eqP-03}
\sum_{i,j=1}^m \sum_{\alpha,\beta=1}^d \Bigabs{A^{\alpha\beta}_{ij}(X)}^2\le \nu^{-2},\quad\forall X\in\bR^{d+1},
\end{equation}
for some constant $\nu\in (0,1]$.
The adjoint operator $\sLt$ is defined by
\[\sLt \vec u = -\vec u_t-D_\alpha \bigl(\widetilde{\vec A}{}^{\alpha\beta} D_\beta \vec u\bigr),\]
where $\widetilde{\vec A}{}^{\alpha\beta}=\bigl(\vec A^{\beta\alpha}\bigr)^T$; i.e., $\tilde A{}^{\alpha\beta}_{ij}=A^{\beta\alpha}_{ji}$.
Notice that the coefficients $\tilde A{}^{\alpha\beta}_{ij}$ satisfy the conditions \eqref{eqP-02} and \eqref{eqP-03} with the same constant $\nu$.

\subsection{Time-varying $H_1$ domain}
We shall say that $\Omega$ is a time-varying $H_1$ graph domain in $\bR^{d+1}$ if there is a function $f=f(X')=f(t,x')$ from $\bR^d$ to $\bR$ satisfying
\begin{equation}					\label{eq2.01dh}
\abs{f(X')-f(Y')}\leq M\abs{X'-Y'}_{\sP},\quad \forall X',Y'\in \bR^d,
\end{equation}
for some constant $M>0$ so that $\Omega$ is represented by
\[
\Omega=\set{X=(X',x^d)\in \bR^{d+1}\colon\, x^d>f(X')}.
\]
We shall say that $\Omega$ is a time-varying $H_1$ domain in $\bR^{d+1}$ if 
\begin{enumerate}[i)]
\item
$I(\Omega)=\bR$ and $\omega(t)$ is a bounded domain in $\bR^d$ for all $t\in \bR$.
\item
There are constants $M$ and $R_a>0$ such that for each $X_0\in \partial\Omega$, there is a function $f=f(X')=f(t,x')$ satisfying \eqref{eq2.01dh}, for which (after a suitable rotation of $x$-axes)
\[
\Omega \cap Q(X_0,R_a)=\set{X \in Q(X_0,R_a)\colon\,  x^d> f(X')}.
\]
\end{enumerate}

\subsection{Function spaces}
For $q\geq 1$, we let $L_q(\Omega)$ denote the classical Banach space consisting of measurable functions on $\Omega$ that are $q$-integrable.
The space $W^{0,1}_q(\Omega)$ denotes the set of functions $u\in L_q(\Omega)$ with its weak derivative $D u \in L_q(\Omega)$ having a finite norm
\[
\norm{u}_{W^{0,1}_q(\Omega)}=\norm{u}_{L_q(\Omega)}+\norm{Du}_{L_q(\Omega)}.
\]
We denote by $W^{1,1}_2(\Omega)$ the Hilbert space with the inner product
\[
\ip{u,v}_{W^{1,1}_2(\Omega)}:=\int_{\Omega} uv+\sum_{\alpha=1}^d \int_{\Omega} D_\alpha u D_\alpha v+\int_{\Omega} u_t v_t.
\]
We define $V_2(\Omega)$ as the set of all $u\in W^{0,1}_2(\Omega)$ having a finite norm $\norm{u}_{V_2(\Omega)}$ defined by
\[
\norm{u}_{V_2(\Omega)}^2 :=\int_\Omega \abs{Du}^2 \,dX + \esssup_{t\in I(\Omega)} \int_{\omega(t)} u^2 \,dx.
\]
The space $V^{0,1}_2(\Omega)$ is obtained by completing the set $W^{1,1}_2(\Omega)$ in the norm of $V_2(\Omega)$.
Let $\varSigma \subset \overline \Omega$ and $u$ be a $V^{0,1}_2(\Omega)$ function.
We say that $u$ vanishes (or write $u=0$) on $\varSigma$ if $u$ is a limit in $V^{0,1}_2(\Omega)$ of a sequence of functions in $C^\infty_c(\overline \Omega\setminus \varSigma)$.
We define $\rV^{0,1}_2(\Omega)$ to be the set of all functions $u$ in $V^{0,1}_2(\Omega)$ that vanishes on $S\Omega$.

\subsection{Weak Solutions}
For $\vec f, \vec g_\alpha \in L_{1,loc}(\Omega)^m$ ($\alpha=1,\ldots,d$), we say that $\vec u$ is a weak solution of $\sL \vec u=\vec f+ D_\alpha\vec g_\alpha$ in $\Omega$ if $\vec u \in V_2(\Omega)^m$ and satisfies
\begin{equation}					\label{eqn:E-71}
-\int_{\Omega} u^i\phi^i_t+ \int_{\Omega} A^{\alpha\beta}_{ij}D_\beta u^j D_\alpha\phi^i= \int_{\Omega} f^i \phi^i- \int_{\Omega} g^i_\alpha D_\alpha\phi^i, \quad\forall \vec \phi \in C^\infty_c (\Omega)^m.
\end{equation}
We say that $\vec u$ is a weak solution of $\sLt \vec u=\vec f+D_\alpha\vec g_\alpha$ in $\Omega$ if $\vec u\in V_2(\Omega)^m$ and satisfies
\begin{equation}					\label{eqn:E-71b}
\int_{\Omega} u^i\phi^i_t+ \int_{\Omega}\tilde A{}^{\alpha\beta}_{ij}D_\beta u^j D_\alpha\phi^i=
\int_{\Omega} f^i \phi^i- \int_{\Omega} g^i_\alpha D_\alpha\phi^i, \quad\forall \vec\phi \in C^\infty_c (\Omega)^m.
\end{equation}
For $\vec\psi_0=\vec \psi_0(x) \in L_{1,loc}(\omega(t_0))^m$, we say that $\vec u$ is a weak solution of the problem
\[
\sL \vec u =\vec f + D_\alpha \vec g_\alpha\;\text{ in }\;\Omega(t_0,t_1),\quad \vec u=0\;\text{ on }\; S\Omega(t_0,t_1),\quad \vec u=\vec\psi_0\;\text{ on }\;\omega(t_0)
\]
if $\vec u \in \rV_2^{0,1}(\Omega(t_0,t_1))$ and satisfies for all $\tau\in I(\Omega(t_0,t_1))$ the identity
\begin{multline*}
\int_{\omega(\tau)} u^i \phi^i\,dx - \int_{\Omega(t_0,\tau)} u^i\phi^i_t\,dX+\int_ {\Omega(t_0,\tau)} A^{\alpha\beta}_{ij}D_\beta u^j D_\alpha\phi^i\,dX=\int_ {\Omega(t_0,\tau)} f^i\phi^i\,dX\\
- \int_{\Omega(t_0,\tau)} g^i_\alpha D_\alpha\phi^i\,dX+\int_{\omega(t_0)} \psi_0^i \phi^i\,dx,\quad\forall \vec \phi \in C^\infty_c(\overline{\Omega(t_0,t_1)}\setminus S\Omega(t_0,t_1))^m.
\end{multline*}
Similarly, for $\vec\psi_0=\vec \psi_0(x) \in L_{1,loc}(\omega(t_1))^m$, we say that $\vec u$ is a weak solution of the problem
\[
\sLt \vec u =\vec f + D_\alpha \vec g_\alpha\;\text{ in }\;\Omega(t_0,t_1),\quad \vec u=0\;\text{ on }\; S\Omega(t_0,t_1),\quad \vec u=\vec\psi_0\;\text{ on }\;\omega(t_1)
\]
if $\vec u \in \rV_2^{0,1}(\Omega(t_0,t_1))$ and satisfies for all $\tau\in I(\Omega(t_0,t_1))$ the identity
\begin{multline*}
\int_{\omega(\tau)} u^i \phi^i\,dx + \int_{\Omega(\tau,t_1)} u^i\phi^i_t\,dX+\int_ {\Omega(\tau,t_1)} \tilde A{}^{\alpha\beta}_{ij}D_\beta u^j D_\alpha\phi^i\,dX=\int_ {\Omega(\tau,t_1)} f^i\phi^i\,dX\\
- \int_{\Omega(\tau,t_1)} g^i_\alpha D_\alpha\phi^i\,dX+\int_{\omega(t_1)} \psi_0^i \phi^i\,dx,\quad\forall \vec \phi\in C^\infty_c(\overline{\Omega(t_0,t_1)}\setminus S\Omega(t_0,t_1))^m.
\end{multline*}

\subsection{Green's function}					\label{sec:dgf}
Let $\Omega$ be a time-varying $H_1$ (graph) domain in $\bR^{d+1}$.
We say that an $m\times m$ matrix valued function $\vec G(X,Y)=\vec G(t,x,s,y)$, with entries $G_{ij} (X,Y)$ defined on the set $\bigset{(X,Y)\in\Omega\times\Omega: X\neq Y}$, is a Green's function of $\sL$ in $\Omega$ if it satisfies the following properties:
\begin{enumerate}[i)]
\item
$\vec G(\cdot,Y)\in W^{0,1}_{1,loc}(\Omega)$ and $\sL \vec G(\cdot,Y) = \delta_Y I_m$ for all $Y\in \Omega$, in the sense that
\[
\int_{\Omega}\left( -G_{ik}(\cdot,Y)\phi^i_t+ A^{\alpha\beta}_{ij} D_\beta G_{jk}(\cdot,Y)D_\alpha \phi^i\right) = \phi^k(Y), \quad \forall \vec \phi \in C^\infty_c(\Omega)^m.
\]
\item
$\vec G(\cdot,Y) \in V_2^{0,1}(\Omega\setminus Q(Y,R))$ for all $Y\in\Omega$ and $R>0$, and $\vec G(\cdot,Y)$ vanishes on $S\Omega$.
\item
For any $\vec f=(f^1,\ldots, f^m)^T \in C^\infty_c(\Omega)$, the function $\vec u$ given by
\[
\vec u(X):=\int_\Omega \vec G(Y,X) \vec f(Y)\,dY
\]
belongs to $\rV^{0,1}_2(\Omega)$ and satisfies $\sLt \vec u=\vec f$ in the sense of \eqref{eqn:E-71b}.
\end{enumerate}
Similarly, we say that an $m\times m$ matrix valued function $\tilde{\vec G}(X,Y)=\tilde{\vec G}(t,x,s,y)$ is a Green's function of $\sLt$ in $\Omega$ if it satisfies the following properties:
\begin{enumerate}[i)]
\item
$\tilde{\vec G}(\cdot,Y)\in W^{0,1}_{1,loc}(\Omega)$ and $\sLt \tilde{\vec G}(\cdot,Y) = \delta_Y I_m$ for all $Y\in \Omega$, in the sense that
\[
\int_{\Omega}\left(\tilde G_{ik}(\cdot,Y)\phi^i_t+ \tilde A^{\alpha\beta}_{ij} D_\beta \tilde G_{jk}(\cdot,Y)D_\alpha \phi^i\right) = \phi^k(Y), \quad \forall \vec \phi \in C^\infty_c(\Omega)^m.
\]
\item
$\tilde{\vec G}(\cdot,Y) \in V_2^{0,1}(\Omega\setminus Q(Y,R))$ for all $Y\in\Omega$ and $R>0$, and $\tilde{\vec G}(\cdot,Y)$ vanishes on $S\Omega$.
\item
For any $\vec f=(f^1,\ldots, f^m)^T \in C^\infty_c(\Omega)$, the function $\vec u$ given by
\[
\vec u(X):=\int_\Omega \tilde{\vec G}(Y,X) \vec f(Y)\,dY
\]
belongs to $\rV^{0,1}_2(\Omega)$ and satisfies $\sL \vec u=\vec f$ in the sense of \eqref{eqn:E-71}.
\end{enumerate}
We remark that part iii) of the above definitions combined with the uniqueness of weak solutions of $\sLt \vec u =\vec f$ and $\sL \vec u =\vec f$ in $\rV_2^{0,1}(\Omega)$ for any $\vec f\in C^\infty_c(\Omega)$ gives uniqueness of Green's functions; see \cite[\S 3.6]{CDK} and \cite{BHL}.

\mysection{Main results} \label{main}
The following condition $\IH$ means that weak solutions of $\sL\vec u=0$ and $\sLt \vec u=0$ enjoy interior H\"older continuity estimates with an exponent $\mu_0$.
It is not hard to see that this condition is equivalent to saying that the operator $\sL$ and its adjoint $\sLt$ satisfy the property $\PH$ in \cite{CDK}; see \cite[Lemma~8.2]{CDK2} for the proof.

\begin{CIH}
There exist constants $\mu_0\in (0,1]$, $R_c \in (0,\infty]$, and  $C_0>0$ such that for all $X\in\Omega$ the following holds:
\begin{enumerate}[i)]
\item
If  $\vec u$ is a weak solution of $\sL \vec u=0$ in $Q_{-}(X,R)$, where $R<R_c\wedge d^{-}(X)$, then we have
\[
[\vec u]_{\mu_0/2,\mu_0;Q_{-}(X,R/2)}\leq C_0 R^{-\mu_0}\left(\fint_{Q_{-}(X,R)} \abs{\vec u}^2\right)^{1/2}.
\]
\item
If  $\vec u$ is a weak solution of $\sLt \vec u=0$ in $Q_{+}(X,R)$, where $R<R_c\wedge d^{+}(X)$, then we have
\[
[\vec u]_{\mu_0/2,\mu_0;Q_{+}(X,R/2)}\leq C_0 R^{-\mu_0}\left(\fint_{Q_{+}(X,R)} \abs{\vec u}^2\right)^{1/2}.
\]
\end{enumerate}
\end{CIH}

By assuming the condition $\IH$, we construct the Green's function of $\sL$ in time-varying $H_1$ domains and the domains above time-varying $H_1$ graph in $\bR^{d+1}$.

\begin{theorem}	\label{thm1}
Let $\Omega$ be a time-varying $H_1$ (graph) domain in $\bR^{d+1}$.
Assume the condition $\IH$.
Then there exists a unique Green's function $\vec G(X,Y)=\vec G(t,x,s,y)$ of $\sL$ in $\Omega$.
We have $\vec G(\cdot,Y) \in C^{\mu_0/2,\mu_0}_{loc}(\Omega\setminus \set{Y})$ for all $Y\in \Omega$ and 
\begin{equation}					\label{eq3.00ow}
\vec G(\cdot,Y)\equiv 0\quad\text{on }\;\Omega(-\infty,s).
\end{equation}
Also, there exists a unique Green's function $\tilde{\vec G}(X,Y)$ of $\sLt$ in $\Omega$, which satisfies
\begin{equation}					\label{eq3.02wb}
\tilde{\vec G}(\cdot,Y)\equiv 0\quad\text{on }\;\Omega(s,\infty)
\end{equation}
and  $\tilde{\vec G}(\cdot,Y) \in C^{\mu_0/2,\mu_0}_{loc}(\Omega\setminus \set{Y})$ for all $Y\in\Omega$.
In addition, we have the following the identity
\begin{equation}					\label{eq3.01mq}
\tilde{\vec G}(X,Y):=\vec G(Y,X)^T,\quad \forall X, Y \in \Omega,\;\; X\neq Y.
\end{equation}
Moreover, for any $\vec \psi_0\in L_2(\omega(s_0))^m$, the function $\vec u(t,x)$ given by
\begin{equation}  \label{eq5.24}
\vec{u}(t,x)=\int_{\omega(s_0)} \vec G(t,x,s_0,y)\vec \psi_0(y)\,dy, \quad \forall X=(t,x)\in\Omega(s_0,\infty),
\end{equation}
is a unique weak solution of the problem
\begin{equation} \label{eq5.88}
\sL \vec u=0\;\text{ in }\;\Omega(s_0,\infty),\quad \vec u=0\;\text{ on }\; S\Omega(s_0,\infty),\quad \vec u=\vec\psi_0\;\text{ on }\;\omega(s_0)
\end{equation}
and if $\vec \psi_0$ is continuous at $x_0\in\omega(s_0)$ in addition, then
\begin{equation}					\label{eq5.23}
\lim_{\substack{(t,x)\to (s_0,x_0)\\X\in\Omega(s_0,\infty)}} \int_{\omega(s_0)}\vec G(t,x,s_0,y)\vec \psi_0(y)\,dy =\vec \psi_0(x_0).
\end{equation}
Furthermore, the following estimates hold for $\vec G$, where we denote $d_Y'=d(Y)\wedge R_c$:
\begin{enumerate}[i)]
\item
$\norm{\vec G(\cdot,Y)}_{L_{2+4/d}(\Omega\setminus \overline Q(Y,R))}+ \norm{\vec G(\cdot,Y)}_{V_2(\Omega\setminus \overline Q(Y,R))} \leq N R^{-d/2}\;$ for all $R<d_Y'$ and $Y\in\Omega$.
\item
$\norm{\vec G(\cdot,Y)}_{L_p(\Omega[Y,R])} \leq N R^{-d+(d+2)/p}\;$ for all $r<d_Y'$,   $Y\in\Omega$, and $p \in \bigl[1,\frac{d+2}{d}\bigr)$.
\item
$\bigabs{\set{X\in \Omega\colon\, \abs{\vec G(X,Y)}>\tau}} \leq N \tau^{-(d+2)/d}\;$ for all $\tau>(d_Y'/2)^{-d}$ and $Y\in\Omega$.
\item
$\norm{D\vec G(\cdot,Y)}_{L_p(\Omega[Y,R])} \leq N R^{-d-1+(d+2)/p}\;$ for all $r<d_Y'$, $Y\in\Omega$, and $p\in \bigl[1,\frac{d+2}{d+1}\bigr)$ .
\item
$\bigabs{\set{X\in \Omega\colon\,\abs{D_x\vec G(X,Y)}>\tau}} \leq N \tau^{-(d+2)/(d+1)}\;$ for all $\tau>(d_Y'/2)^{-d}$ and $Y\in\Omega$.
\item
$\abs{\vec G(X,Y)}\leq C \abs{X-Y}_{\sP}^{-d}\;$ whenever $0<\abs{X-Y}_{\sP}<d_Y'/2$ and $X, Y\in\Omega$.
\item
$\abs{\vec G(X,Y)-\vec G(X',Y)}\leq C\abs{X-X'}_{\sP}^{\mu_0} \abs{X-Y}_{\sP}^{-d-\mu_0}\;$ whenever $2\abs{X-X'}_{\sP}<\abs{X-Y}_{\sP}<d_Y'/2$ and $X, X', Y\in\Omega$.
\end{enumerate}
In the above, $N=N(d,m,\nu,\mu_0,C_0)$ and $N$ depends on $p$ as well in ii) and iv).
The estimates i) -- vii) are also valid for the Green's function $\tilde{\vec G}$ of the adjoint operator $\sLt$ in $\Omega$.
\end{theorem}

\begin{remark}					\label{rmk2.6}
In the condition $\IH$, the constant $R_c$ is interchangeable with $a R_c$ for any fixed $a \in (0,\infty)$, possibly at the cost of increasing the constant $C_0$.
Also, the condition $\IH$ implies that if $\vec u$ is a weak solution of $\sL \vec u=0$ in $Q_{-}(X_0,R)$ with $R<d^{-}(Y)\wedge R_c$, then we have the $L_\infty$ estimate
\begin{equation}					\label{eq3.10np}
\norm{\vec u}_{L_\infty(Q_{-}(X_0,R/4))} \leq N \left(\fint_{Q_{-}(X_0,R)} \abs{\vec u}^2\right)^{1/2},
\end{equation}
where $N=N(d,m,\nu,\mu_0, C_0)>0$.
Moreover, $\vec u$ satisfy
\[
\norm{\vec u}_{L_\infty(Q_{-}(X_0,r))} \leq N  (R-r)^{-(d+2)/p} \norm{\vec u}_{L_p(Q_{-}(X_0,R))},\quad \forall r <R, \;\; \forall p>0, 
\]
where $N=N(d,m,\nu,\mu_0, C_0,p)>0$.
See \cite[Lemma~2.6]{CDK} for the proof.
\end{remark}

\begin{remark}
In Theorem~\ref{thm1}, we also have the following estimates, which follow from the identity \eqref{eq3.01mq} and the estimates \textit{i) -- vi)} for $\tilde{\vec G}(\cdot, X)$:
\begin{enumerate}[i)]
\em
\item
\em
$\norm{\vec G(X,\cdot)}_{L_{2+4/d}(\Omega\setminus \overline Q(X,R))}+ \norm{\vec G(X,\cdot)}_{V_2(\Omega\setminus \overline Q(X,R))} \leq N R^{-d/2}\;$ for all $R<d_X'$ and $X\in\Omega$.
\em
\item
\em
$\norm{\vec G(X,\cdot)}_{L_p(\Omega[X,R])} \leq N R^{-d+(d+2)/p}\;$ for all $r<d_X'$,   $X\in\Omega$, and $p \in \bigl[1,\frac{d+2}{d}\bigr)$.
\em
\item
\em
$\bigabs{\set{Y\in \Omega\colon\,\abs{\vec G(X,Y)}>\tau}} \leq N \tau^{-(d+2)/d}\;$ for all $\tau>(d_X'/2)^{-d}$ and $X\in\Omega$.
\em
\item
\em
$\norm{D\vec G(X,\cdot)}_{L_p(\Omega[X,R])} \leq N R^{-d-1+(d+2)/p}\;$ for all $r<d_X'$, $X\in\Omega$, and $p\in \bigl[1,\frac{d+2}{d+1}\bigr)$ .
\em
\item
\em
$\bigabs{\set{Y\in \Omega\colon\,\abs{D_y\vec G(X,Y)}>\tau}} \leq N \tau^{-(d+2)/(d+1)}\;$ for all $\tau>(d_X'/2)^{-d}$ and $X\in\Omega$.
\em
\item
\em
$\abs{\vec G(X,Y)}\leq C \abs{X-Y}_{\sP}^{-d}\;$ whenever $0<\abs{X-Y}_{\sP}<d_X'/2$ and $X, Y\in\Omega$.
\em
\item
\em
$\abs{\vec G(X,Y)-\vec G(X,Y')}\leq C\abs{Y-Y'}_{\sP}^{\mu_0} \abs{X-Y}_{\sP}^{-d-\mu_0}\;$ whenever $2\abs{Y-Y'}_{\sP}<\abs{X-Y}_{\sP}<d_X'/2$ and $X, Y, Y'\in\Omega$.
\end{enumerate}
In particular, $\abs{\vec G(X,Y)}\leq N \abs{X-Y}_{\sP}^{-d}$
whenever $0<\abs{X-Y}_{\sP}<\frac{1}{2}(d(X) \vee d(Y))\wedge R_c$.
\end{remark}

The following condition $\LB$ is used to obtain a global Gaussian bound for the Green's function $\vec G(X,Y)$ in a time-varying $H_1$ domain $\Omega\subset \bR^{d+1}$.

\begin{CLB}
There exist constants $R_{max}\in (0,\infty]$ and $N_0>0$ so that for all $X\in\Omega$ and $0<R<R_{max}$, the following holds.
\begin{enumerate}[i)]
\item
If $\vec u$ is a weak solution of $\sL \vec u=0$ in $\Omega_{-}[X,R]$ vanishing on $\sP \Omega_{-}[X,R]$, then we have
\[
\norm{\vec u}_{L_\infty(\Omega_{-}[X,R/2])} \leq N_0 R^{-(2+d)/2} \norm{\vec u}_{L_2(\Omega_{-}[X,R])}.
\]
\item
If $\vec u$ is a weak solution of $\sLt \vec u=0$ in $\Omega_{+}[X,R]$ vanishing on $\widetilde \sP\Omega_{+}[X,R]$, then we have
\[
\norm{\vec u}_{L_\infty(\Omega_{+}[X,R/2])} \leq N_0 R^{-(2+d)/2} \norm{\vec u}_{L_2(\Omega_{+}[X,R])}.
\]
\end{enumerate}
\end{CLB}

\begin{theorem}	\label{thm2}
Let $\Omega$ be a time-varying $H_1$ (graph) domain in $\bR^{d+1}$.
Assume the condition $\LB$ as well as the condition $\IH$.
Then the Green's function $\vec G(X,Y)$ of $\sL$ in $\Omega$ exists and satisfies the conclusions of Theorem~\ref{thm1}.
Moreover, for all $X=(t,x)$ and $Y=(s,y) $ in $\Omega$ with $t>s$, we have
\begin{equation} \label{eq2.17d}
\abs{\vec G(t,x,s,y)}\leq N \left\{(t-s)\wedge R_{max}^2\right\}^{-d/2}\exp\left\{-\kappa \abs{x-y}^2/(t-s)\right\},
\end{equation}
where $N= N(d,m,\nu, N_0)$ and $\kappa=\kappa(\nu)>0$.
\end{theorem}

\begin{remark}					\label{rmk3.15}
In the condition $\LB$, the constant $R_{max}$ is interchangeable with $a R_{max}$ for any fixed $a \in (0,\infty)$, possibly at the cost of increasing the constant $N_0$.
In Theorem~\ref{thm2}, the estimate \eqref{eq2.17d} implies, via straightforward computation, that
\begin{equation}					\label{eq5.18}
\abs{\vec G(X,Y)} \leq N \abs{X-Y}_\sP^{-d},\quad \text{if }\, 0<\abs{t-s} <R_{max}^2,
\end{equation}
where $N=N(d,m,\nu,N_0)$.
Then, similar to Lemma~\ref{lem4.2} below, one can show
\begin{equation}					\label{eq5.15}
\norm{\vec G(\cdot,Y)}_{L_{2+4/d}(\Omega\setminus \overline Q(Y,R))}+\norm{\vec G(\cdot,Y)}_{V_2(\Omega\setminus \overline Q(Y,R))} \leq  NR^{-d/2}, \quad \forall R \in (0,R_{max}),
\end{equation}
where $N=N(d,m,\nu,N_0)$.
Moreover, using \eqref{eq5.15} and proceeding as in \cite[Section~4.2]{CDK}, one can show that $\vec G$ satisfies the estimates {\em ii) -- vi)} in Theorem~\ref{thm1} with $d_Y'$ replaced by $R_{max}$.
Also, it is clear that the estimate \eqref{eq0.2} in the introduction follows from Theorem~\ref{thm2}.
\end{remark}

In order to derive the estimate \eqref{eq0.3} in the introduction, we introduce the following condition $\LH$ which, loosely speaking, says that weak solutions of $\sL u=0$ and $\sLt u=0$ vanishing on $\varSigma\subset \partial\Omega$ are locally H\"older continuous up to $\varSigma$ with exponent $\mu_0$.

\begin{CLH}
There exist $\mu_0\in (0,1]$, $R_{max} \in (0,\infty]$, and  $N_1>0$ so that for all $X\in \Omega$ and $0<R<R_{max}$, the following holds.
\begin{enumerate}[i)]
\item
If  $\vec u$ is a weak solution of $\sL \vec u=0$ in $\Omega_{-}[X,R]$ vanishing on $\sP\Omega_{-}[X,R]$, then we have
\[
[\tilde{\vec u}]_{\mu_0/2,\mu_0;Q_{-}(X,R/2)}\leq N_1 R^{-\mu_0}\left(\fint_{Q_{-}(X,R)} \abs{\tilde{\vec u}}^2\right)^{1/2},\;\text{where }\;\tilde{\vec u}=\chi_{\Omega_{-}[X,R]} \vec u.
\]
\item
If  $\vec u$ is a weak solution of $\sLt \vec u=0$ in $\Omega_{+}[X,R]$ vanishing on $\widetilde \sP \Omega_{+}[X,R]$, then we have
\[
[\tilde{\vec u}]_{\mu_0/2,\mu_0;Q_{+}(X,R/2)}\leq N_1 R^{-\mu_0}\left(\fint_{Q_{+}(X,R)} \abs{\tilde{\vec u}}^2\right)^{1/2},\;\text{ where }\;\tilde{\vec u}=\chi_{\Omega_{+}[X,R]} \vec u.
\]
\end{enumerate}
\end{CLH}

It is easy to see that the condition $\LH$ implies the condition $\LB$; see Lemma~\ref{lem2.19} in Appendix for the proof.
Also, it is obvious that the condition $\LH$ implies the condition $\IH$.
Therefore if the condition $\LH$ is satisfied, then there exists the Green's function of $\sL$ and it satisfies the conclusions of Theorems \ref{thm1} and \ref{thm2}.
The following theorem says that in fact, in such a case, a better estimate for the Green's function is available.

\begin{theorem}			\label{thm3}
Let $\Omega$ be a time-varying $H_1$ (graph) domain in $\bR^{d+1}$.
Assume the condition $\LH$.
Then the Green's function $\vec G(X,Y)$ of $\sL$ in $\Omega$ exists and satisfies the conclusions of Theorem~\ref{thm1}.
Moreover, for all $X=(t,x)$ and $Y=(y,s) $ in $\Omega$ with $t>s$, we have
\begin{equation}					\label{eq3.8yy}
\abs{\vec G(t,x,s,y)}\leq N\delta(X,Y)^{\mu_0} \{(t-s)\wedge R_{max}^2\}^{-d/2}\exp\left\{-\kappa \abs{x-y}^2/(t-s) \right\},
\end{equation}
where $N=N(d,m,\nu, \mu_0,N_1)$ and $\kappa=\kappa(\nu)>0$ and we used the  notation
\begin{equation}			\label{eq3.6mm}
\delta(X,Y)= \left(1 \wedge \frac {d^{-}(X)} {R_{max}\wedge \abs{X-Y}_\sP} \right) \left(1 \wedge \frac{d^{+}(Y)} {R_{max}\wedge \abs{X-Y}_\sP}\right).
\end{equation}
\end{theorem}

\begin{remark}					\label{rmk3.19}
In the condition $\LH$, the constant $R_{max}$ is interchangeable with $a R_{max}$ for any fixed $a \in (0,\infty)$, possibly at the cost of increasing the constant $N_1$.
Also, we note that the estimate \eqref{eq0.3} in the introduction follows from Theorem~\ref{thm3} if $\Omega$ be a time-varying $H_1$ domain or a time-varying $H_1$ graph domain with $R_{max}=\infty$.
\end{remark}

\begin{remark}					\label{rmk3.20}
In Theorem~\ref{thm3}, we also have the estimate
\[
\abs{\vec G(X,Y)-\vec G(X',Y)}\leq \frac{N\delta(X,Y)^{\mu_0}}{\left\{(t-s)\wedge R_{max}^2\right\}^{d/2}} \left(\frac{\abs{X-X'}_{\sP}}{\abs{X-Y}_{\sP}}\right)^{\mu_0}\exp\left\{-\frac{\kappa \abs{x-y}^2}{t-s} \right\}
\]
whenever $2\abs{X-X'}_{\sP}< \abs{X-Y}_{\sP}$ and $t>s$.
It follows from \eqref{eq3.8yy} and the condition $\LH$.
\end{remark}

\section{Some Applications of Main Results}				\label{sec:app}
\subsection{Scalar case}
In the scalar case (i.e., $m=1$), both conditions $\LB$ and $\IH$ are satisfied with $R_c=R_{max}=\infty$ and $N_0=N_0(d,\nu)$; see e.g., \cite[Chapter ~VI]{Lieberman}.
Also, in the scalar case, the Green's function is a nonnegative scalar function.
Therefore, the following corollary is an immediate consequence of Theorem~\ref{thm2}.

\begin{corollary}           \label{cor1}
Let $\Omega$ be a time-varying $H_1$ (graph) domain in $\bR^{d+1}$.
If $m=1$, then the Green's function $G(X,Y)$ of $\sL$ in $\Omega$ exists and satisfies the conclusions of Theorem~\ref{thm1} with $d_Y'$ replaced by $R_a$.
Moreover, for all $X=(t,x)$ and $Y=(y,s) $ in $\Omega$ with $t>s$, we have
\[
G(t,x,s,y) \leq N (t-s)^{-d/2}\exp\left\{-\kappa \abs{x-y}^2/(t-s)\right\},
\]
where $N=N(d,\nu)$ and $\kappa=\kappa(\nu)$  are universal constants independent of $\Omega$.
\end{corollary}

In fact, in the scalar case, a better estimate is available near the boundary.
Let $\Omega$ be a time-varying $H_1$ (graph) domain in $\bR^{d+1}$.
By using the results in \cite[\S VI.8]{Lieberman}, one can show that in the case when $m=1$, the condition $\LH$ is satisfied in $\Omega$.
Moreover, in the case when $\Omega$ is a time-varying $H_1$ graph domain, then the condition $\LH$ is satisfied with $R_{max}=\infty$.
Also, in that case, there exists $N=N(M)\geq 1$ such that
\begin{equation}					\label{eq4.00yx}
1 \leq d^{-}(X)/d(X),\; d^{+}(X)/d(X) \leq N,\quad \forall X\in \Omega.
\end{equation}
Therefore, the following corollaries are immediate consequences of Theorem~\ref{thm3}.
\begin{corollary}           \label{cor2}
Assume that $m=1$ and let  $G(X,Y)$ be the Green's function of $\sL$ in $\Omega$, where $\Omega$ is a time-varying $H_1$ domain in $\bR^{d+1}$.
Let $\delta(X,Y)$ be as defined in \eqref{eq3.6mm} with $R_{max}=R_a$.
Then, for all $X=(t,x)$ and $Y=(y,s) $ in $\Omega$ with $t>s$, we have
\[
G(t,x,s,y)\leq N\delta(X,Y)^{\mu_0} \{(t-s)\wedge R_{a}^2\}^{-d/2}\exp\left\{-\kappa \abs{x-y}^2/(t-s) \right\},
\]
where $N=N(d,\nu)$ and $\kappa=\kappa(\nu)$  are positive constants independent of $\Omega$.
\end{corollary}

\begin{corollary}           \label{cor3}
Assume that $m=1$ and let  $G(X,Y)$ be the Green's function of $\sL$ in $\Omega$, where $\Omega$ is a time-varying $H_1$ graph domain in $\bR^{d+1}$.
Then, for all $X=(t,x)$ and $Y=(y,s) $ in $\Omega$ with $t>s$, we have
\[
G(t,x,s,y)\leq N \left(1 \wedge \frac {d(X)} {\abs{X-Y}_\sP} \right)^{\mu_0} \left(1 \wedge \frac{d(Y)} {\abs{X-Y}_\sP}\right)^{\mu_0} \frac{1}{(t-s)^{d/2}}\exp\Set{-\frac{\kappa\abs{x-y}^2}{t-s}},
\]
where $N=N(d,\nu,M)$ and $\kappa=\kappa(\nu)$ are positive constants.
\end{corollary}

\subsection{$L_\infty$-perturbation of diagonal systems}        \label{sec:AD}
Let $a^{\alpha\beta}(X)$ be scalar functions satisfying
\begin{equation}            \label{eqP-07}
a^{\alpha\beta}(X)\xi_\beta\xi_\alpha\ge \nu_0\bigabs{\vec \xi}^2,\;\;\forall\xi\in\bR^d;\quad \sum_{\alpha,\beta=1}^d \bigabs{a^{\alpha\beta}(X)}^2\leq \nu_0^{-2},
\end{equation}
for all $X \in\bR^{d+1}$ with some constant $\nu_0\in (0,1]$.
Let $\Omega$ be a time-varying $H_1$ (graph) domain in $\bR^{d+1}$.
Let $A^{\alpha\beta}_{ij}$ be the coefficients of the operator $\sL$. We denote
\begin{equation}					\label{eqP-08w}
\mathscr{E}= \sup_{X\in \bR^{d+1}}\left\{ \sum_{i,j=1}^m \sum_{\alpha,\beta =1}^d \Bigabs{A^{\alpha\beta}_{ij}(X)-a^{\alpha\beta}(X)\delta_{ij}}^2\right\}^{1/2},
\end{equation}
where $\delta_{ij}$ is the Kronecker delta symbol.
By Lemma~\ref{lem:G-06}, there exists $\mathscr{E}_0=\mathscr{E}_0(d,\nu_0,M)$ such that if $\mathscr{E}<\mathscr{E}_0$, then the condition $\LH$ is satisfied with $\mu_0=\mu_0(d,\nu_0,M)$, $R_{max}=R_a$, and $N_1=N_1(d,m,\nu_0, M)$.
Therefore, the following corollaries are another easy consequences of Theorem~\ref{thm3}.

\begin{corollary}           \label{cor2b}
Let $\Omega$ be a time-varying $H_1$ domain in $\bR^{d+1}$ and let $\delta(X,Y)$ be as in \eqref{eq3.6mm} with $R_{max}=R_a$.
There exists $\mathscr{E}_0=\mathscr{E}_0(d,\nu_0,M)$ such that if $\mathscr{E} <\mathscr{E}_0$, then the Green's function $\vec G(X,Y)$ of $\sL$ in $\Omega$ exists and satisfies the conclusions of Theorem~\ref{thm1} with $d_Y'$ replaced by $R_a$.
Moreover, for all $X=(t,x)$ and $Y=(y,s) $ in $\Omega$ with $t>s$, we have
\[
\abs{\vec G(t,x,s,y)} \leq N \delta(X,Y)^{\mu_0} \{(t-s)\wedge R_a^2\}^{-d/2}\exp\{-\kappa \abs{x-y}^2/(t-s)\},
\]
where $N, \mu_0$, and $\kappa$ are constants depending on $d,m,\nu_0$, and $M$.
\end{corollary}

\begin{corollary}           \label{cor2c}
Let $\Omega$ be a time-varying $H_1$ graph domain in $\bR^{d+1}$.
There exists $\mathscr{E}_0=\mathscr{E}_0(d,\nu_0,M)$ such that if $\mathscr{E} <\mathscr{E}_0$, then the Green's function $\vec G(X,Y)$ of $\sL$ in $\Omega$ exists and satisfies the conclusions of Theorem~\ref{thm1} with $d_Y'$ replaced by $\infty$.
Moreover, for all $X=(t,x)$ and $Y=(y,s) $ in $\Omega$ with $t>s$, we have
\[
\abs{\vec G(t,x,s,y)} \leq N \left(1 \wedge \frac {d(X)} {\abs{X-Y}_\sP} \right)^{\mu_0} \left(1 \wedge \frac{d(Y)} {\abs{X-Y}_\sP}\right)^{\mu_0} \frac{1}{(t-s)^{d/2}}\exp\Set{-\frac{\kappa\abs{x-y}^2}{t-s}},
\]
where $N, \mu_0$, and $\kappa$ are constants depending on $d,m,\nu_0$, and $M$.
\end{corollary}

\subsection{Systems with $\VMO_x$ coefficients}      \label{sec:VMO}
For a measurable function $f=f(X)=f(t,x)$ defined on $\bR^{d+1}$, we set for $\rho>0$
\[
\omegaup_\rho(f):=\sup_{X\in\bR^{d+1}}\sup_{r \leq \rho} \fint_{t-r^2}^{t+r^2}\!\!\!\fint_{B(x,r)} \bigabs{f(y,s)-\bar f_{x,r}(s)}\,dy\,ds;\quad \bar f_{x,r}(s)=\fint_{B(x,r)}f(s,\cdot).
\]
We say that $f$ belongs to $\VMO_x$ if $\lim_{\rho\to 0} \omegaup_\rho(f)=0$.
Note that $\VMO_x$ is a strictly larger class than the classical $\VMO$ space.
In particular, $\VMO_x$ contains all functions uniformly continuous in $x$ and measurable in $t$; see \cite{Krylov}.

By \cite[Lemma~2.3]{CDK}, we find that  if the coefficients of $\sL$ belong to $\VMO_x$, then the condition $\IH$ is satisfied with parameters $\mu_0$, $N_0$, and $R_{c}$ depending on $\omegaup_\rho(\vec A^{\alpha\beta})$ as well as on $d, m, \nu$.
Therefore, we have the following corollary of Theorem~\ref{thm1}.

\begin{corollary}           \label{cor5}
Let $\Omega$ be a time-varying $H_1$ (graph) domain in $\bR^{d+1}$.
If the coefficients of $\sL$ belong to $\VMO_x$, then the Green's function of $\sL$ exists in $\Omega$ and satisfies the conclusions of
Theorem~\ref{thm1} with some $R_c>0$.
\end{corollary}

\begin{remark}
In Corollary \ref{cor5}, instead of assuming $\vec A^{\alpha\beta}\in \VMO_x$, one may assume that $\omegaup_\rho(\vec A^{\alpha\beta})$ is sufficiently small for some $\rho>0$.
Also, if $\Omega$ is a time-varying domain satisfying the hypothesis of Section~2.3 with $f =f(X')=f(t,x')\in H_{1+\alpha}(\bR^d)$ for some $\alpha>0$, then one can show that the condition $\LH$ is satisfied with the parameters $\mu_0$, $N_1$, and $R_{max}<\infty$ depending on $d, m, \nu$, and $\omegaup_\rho(\vec A^{\alpha\beta})$; see \cite{Lieberman} for the definition of the space $H_{1+\alpha}$.
Therefore, in that case, for all $X=(t,x)$ and $Y=(y,s) $ in $\Omega$ with $t>s$, we have
\[
\abs{\vec G(t,x,s,y)} \leq N \delta(X,Y)^{\mu_0} \{(t-s)\wedge R_{max}^2\}^{-d/2}\exp\{-\kappa \abs{x-y}^2/(t-s)\},
\]
where $\delta(X,Y)$ is as in \eqref{eq3.6mm}.
\end{remark}
\section{Proofs of Main Theorems}						\label{sec:p}
\subsection{Proof of Theorem~\ref{thm1}}					\label{pf1}
By following \cite{CDK}, we shall first construct the ``averaged" Green's function of $\sL$ in $\Omega$.
Notice that we have $\partial\Omega=\sP \Omega=\widetilde \sP \Omega=S\Omega$.
The following lemma is used for the construction of the averaged Green's function, which follows essentially from Brown et al. \cite{BHL} and an embedding theorem in \cite[\S II.3]{LSU}, which is also valid for functions in $\rV^{0,1}_2(\Omega(t_0,t_1))$.
We remark that the function space $\rV^{0,1}_2(\Omega)$ coincides with the function space $V_0(\Omega)$ used in \cite{BHL}.

\begin{lemma}					\label{lem4.1}
For $\vec g\in C^\infty_c(\Omega(t_0,t_1))^m$ and $\vec \psi_0\in L_2(\omega(t_0))$, there exists a unique weak solution $\vec v \in \rV^{0,1}_2(\Omega(t_0,t_1))$ of the problem
\[
\sL \vec v =\vec g\;\text{ in }\;\Omega(t_0,t_1),\quad \vec v=0\;\text{ on }\; S\Omega(t_0,t_1),\quad \vec v=\vec\psi_0\;\text{ on }\;\omega(t_0).
\]
Moreover, we have the following energy inequality for the weak solution $\vec v$:
\begin{equation}					\label{eq4.01re}
 \norm{\vec v}_{V_2(\Omega(t_0,t_1))} \leq N \left(\norm{\vec g}_{L_{(2d+4)/(d+4)}(\Omega(t_0,t_1))}+\norm{\vec\psi_0}_{L_2(\omega(t_0))}\right);\quad N=N(d,m,\nu).
\end{equation}
Similarly, for $\vec f\in C^\infty_c(\Omega(t_0,t_1))^m$ and $\vec \psi_1\in L_2(\omega(t_1))$, there exists a unique weak solution $\vec u \in \rV^{0,1}_2(\Omega(t_0,t_1)$ of the problem
\[
\sLt \vec u =\vec f\;\text{ in }\;\Omega(t_0,t_1),\quad \vec u=0\;\text{ on }\; S\Omega(t_0,t_1),\quad \vec u=\vec\psi_1\;\text{ on }\;\omega(t_1)
\]
and $\vec u$ satisfies the following energy inequality:
\[
 \norm{\vec u}_{V_2(\Omega(t_0,t_1))} \leq N \left(\norm{\vec f}_{L_{(2d+4)/(d+4)}(\Omega(t_0,t_1))}+\norm{\vec\psi_1}_{L_2(\omega(t_1))}\right);\quad N=N(d,m,\nu).
\]
Furthermore, we have the identity
\begin{equation}					\label{eq4.02un}
\int_{\Omega(t_0,t_1)} \vec f \cdot \vec v\,dX+\int_{\omega(t_1)}\vec v \cdot \vec \psi_1\,dx=\int_{\Omega(t_0,t_1)} \vec u \cdot \vec g \,dX+\int_{\omega(t_0)} \vec u \cdot \vec \psi_0\,dx.
\end{equation}
\end{lemma}

Let us fix a function $\Phi \in C^\infty_c(\bR^{d+1})$ such that $\Phi$ is supported in $Q_{-}(0,1)$, $0\leq \Phi \leq 2$, and $\int_{\bR^{d+1}} \Phi=1$.
Let $Y=(s,y)\in \Omega$ be fixed but arbitrary.
For $0<\epsilon<d(Y)$, we define
\[
\Phi_\epsilon(X)=\Phi_\epsilon(t,x)=\epsilon^{-d-2}\Phi((t-s)/\epsilon^2,(x-y)/\epsilon).
\]
Fix $t_0\in (-\infty,s-\epsilon^2)$ and let $\vec v=\vec v_{\epsilon, Y, k}$ be a unique weak solution of the problem
\[
\sL \vec v = \Phi_\epsilon \vec e_k\;\text{ in }\;\Omega(t_0,\infty),\quad \vec v=0\;\text{ on }\; S\Omega(t_0,\infty),\quad \vec v=0\;\text{ on }\;\omega(t_0),
\]
where $\vec e_k$ is the $k$-th unit vector in $\bR^m$.
By the uniqueness, we find that $\vec v$ does not depend on the particular choice of $t_0$ and we may extend $\vec v$ to the entire $\Omega$ by setting
\begin{equation}					\label{eq11.54}
\vec v=\vec v_{\epsilon,Y, k} \equiv 0\quad\text{on}\quad \Omega(-\infty,s-\epsilon^2).
\end{equation}
Then $\vec v \in \rV^{0,1}_2(\Omega)$ and satisfies for all $\tau > s$ the identity
\[
\int_{\omega(\tau)} v^i \phi^i\,dx - \int_{\Omega(-\infty,\tau)} v^i\phi^i_t\,dX+\int_{\Omega(-\infty,\tau)} A^{\alpha\beta}_{ij}D_\beta v^j D_\alpha\phi^i\,dX =\int_ {Q_{-}(Y,\epsilon)} \Phi_\epsilon \phi^k\,dX
\]
for all $\vec \phi \in C^\infty_c(\Omega)^m$.
We define the averaged Green's function $\vec G^\epsilon(\cdot,Y)=(G^\epsilon_{jk}(\cdot,Y))_{j,k=1}^m$ of the operator $\sL$ in $\Omega$ by
\[
G^\epsilon_{jk}(\cdot,Y)=v^j=v^j_{\epsilon;Y,k}.
\]
Notice that by Lemma~\ref{lem4.1} and an embedding theorem (see \cite[\S II.3]{LSU}) we obtain
\begin{equation}					\label{eq4.05nc}
\norm{\vec G^\epsilon(\cdot,Y)}_{L_{2+4/d}(\Omega)}\leq N \norm{\vec G^\epsilon(\cdot,Y)}_{V_2(\Omega)} \leq N \norm{\vec \Phi_\epsilon}_{L_{(2d+4)/(d+4)}(\Omega)} \leq N \epsilon^{-d(d+2)/(2d+4)}.
\end{equation}

Next, for any given $\vec f\in C^\infty_c(\Omega)^m$, fix $t_1$ such that $\vec f\equiv 0$ on $\Omega(t_1,\infty)$ and let $\vec u$ be a unique weak solution of the problem
\[
\sLt \vec u = \vec f\;\text{ in }\;\Omega(-\infty,t_1),\quad \vec u=0\;\text{ on }\; S\Omega(-\infty,t_1),\quad \vec u=0\;\text{ on }\;\omega(t_1),
\]
Again, by the uniqueness we may extend $\vec u$ to the entire $\Omega$ by setting $\vec u\equiv 0$ on $\Omega(t_1,\infty)$.
Then, $\vec u\in \rV^{0,1}_2(\Omega)$ and satisfies for all $\tau$ the identity
\[
\int_{\omega(\tau)} u^i \phi^i\,dx + \int_{\Omega(\tau,\infty)} u^i\phi^i_t\,dX+\int_{\Omega(\tau,\infty)} \tilde A^{\alpha\beta}_{ij}D_\beta v^j D_\alpha\phi^i\,dX =\int_{\Omega(\tau,\infty)}  f^i \phi^i\,dX
\]
for all $\vec \phi\in C^\infty_c(\Omega)$.
Also, similar to \eqref{eq4.05nc}, we have
\[
\norm{\vec u}_{L_{2+4/d}(\Omega)}\leq N \norm{\vec f}_{L_{(2d+4)/(d+4)}(\Omega)}.
\]
Now, let $X_0\in \Omega$ and $R< d(X_0)\wedge R_c$ be fixed but arbitrary, and assume that $\vec f$ is supported in $Q_{+}(X_0,R)\subset\Omega$.
By using the condition $\IH$ and following the same argument as in \cite[Section~3.2]{CDK}, we obtain
\begin{equation}					\label{eq2.17}
\norm{\vec u}_{L_\infty(Q_{+}(X_0,R/4))} \leq N R^{2-(d+2)/p} \norm{\vec f}_{L_p(Q_{+}(X_0,R))},\quad \forall p>(d+2)/2.
\end{equation}
If $Q_{-}(Y,\epsilon)\subset Q_{+}(X_0,R/4)$, then \eqref{eq4.02un} together with \eqref{eq2.17} yields
\[
\Abs{\int_{Q_{+}(X_0,R)}\vec G^\epsilon(\cdot,Y) \vec f} \leq \int_{Q_{-}(Y,\epsilon)} \Phi_\epsilon \abs{\vec u}\leq NR^{2-(d+2)/p}\norm{\vec f}_{L_p(Q_{+}(X_0,R))}, \quad \forall p>(d+2)/2.
\end{equation*}
By duality, it follows that if $Q_{-}(Y,\epsilon)\subset Q_{+}(X_0,R/4)$, then we have
\[
\norm{\vec G^\epsilon(\cdot, Y)}_{L_q(Q_{+}(X_0,R))}\leq N R^{-d+(d+2)/q},\quad \forall q\in [1, (d+2)/d).
\]
Then, by following the proof of \cite[Lemma~3.2]{CDK}, we conclude
\begin{equation}					\label{eq11.16p}
\abs{\vec G^\epsilon(X,Y)}\leq N \abs{X-Y}_{\sP}^{-d} ,\quad \forall\epsilon\leq \tfrac{1}{3}\abs{X-Y}_{\sP} \;\text{ if }\;\abs{X-Y}_{\sP}<\tfrac{1}{2}(d(Y)\wedge R_c).
\end{equation}

The following lemma is a consequence of the energy inequality of Brown et al. \cite{BHL}, the above estimate \eqref{eq11.16p}, and an embedding theorem in \cite[\S II.3]{LSU}.
\begin{lemma}						\label{lem4.2}
For $R< \frac{1}{2}(d(Y)\wedge R_c)$, let $\zeta\in C^\infty_c(Q(Y,R))$ be a cut-off function such that $0\leq \zeta \leq 1$ and $\zeta=1$ on $Q(Y,R/2)$.
Then, for all $\epsilon>0$ we have
\[
\norm{(1-\zeta)\vec G^\epsilon(\cdot,Y)}_{V_2(\Omega)} \leq N\bigl(\norm{D\zeta}_{L_\infty}^2+\norm{\zeta_t}_{L_\infty}\bigr)^{1/2}R^{1-d/2}.
\]
In particular, for all $\epsilon>0$ and $R< \frac{1}{2}(d(Y)\wedge R_c)$, we have
\[
\norm{\vec G^\epsilon(\cdot,Y)}_{V_2(\Omega\setminus \overline Q(Y,R))} \leq N R^{-d/2}.
\]
\end{lemma}

The following lemma is an analogue of \cite[Lemma~6.1]{CDK} in time-varying $H^1$ domains, the proof of which is essentially the same.
\begin{lemma}					\label{lem4.2p}
Let $\set{u_k}_{k=1}^\infty$ be a sequence in $V_2(\Omega)$.
If $\sup_k \norm{u_k}_{V_2(\Omega)}\leq N<\infty$, then there exists a subsequence $\set{u_{k_j}}_{j=1}^\infty \subseteq \set{u_k}_{k=1}^\infty$ and $u\in V_2(\Omega)$ with $\norm{u}_{V_2(\Omega)}\leq N$ such that $u_{k_j} \rightharpoonup u$ weakly in  $W^{0,1}_2(\Omega(t_0,t_1))$ for all $-\infty<t_0<t_1<\infty$.
Moreover, if all $u_k$ vanish on $S\Omega$, then $u$ also vanishes on $S\Omega$.
\end{lemma}

The above two lemmas contain all ingredients needed for the construction of a Green's function.
By following the argument in \cite[Section~4.2]{CDK} verbatim, we construct the Green's function $\vec G(\cdot,Y)$ from $\vec G^\epsilon(\cdot,Y)$, and it is readily seen that $\vec G(\cdot,Y)\in C^{\mu_0/2,\mu_0}_{loc}(\Omega\setminus \set{Y})$ satisfies \eqref{eq3.00ow} as well as the estimates \textit{i) -- vi)}.
The estimate \textit{vii)} does not appear explicitly in \cite{CDK} but it easily follows from the estimates \textit{vi)} and the condition $\IH$; see \cite[\S 3.6]{HK07}.

Also, fix a function $\Psi \in C^\infty_c(\bR^{d+1})$ such that $\Psi$ is supported in $Q_{+}(0,1)$, $0\leq \Psi \leq 2$, and $\int_{\bR^{d+1}} \Psi=1$.
For $0<\epsilon<d(Y)$, where $Y=(s,y)\in \Omega$ be fixed but arbitrary, we set
\[
\Psi_\epsilon(X)=\Psi_\epsilon(t,x)=\epsilon^{-d-2}\Psi((t-s)/\epsilon^2,(x-y)/\epsilon).
\]
Fix $t_1\in (t+\epsilon^2,\infty)$ and let $\vec w=\vec w_{\epsilon, Y, k}$ be a unique weak solution of the problem
\[
\sLt \vec w = \Psi_\epsilon \vec e_k\;\text{ in }\;\Omega(-\infty,t_1),\quad \vec w=0\;\text{ on }\; S\Omega(-\infty,t_1),\quad \vec w=0\;\text{ on }\;\omega(t_1),
\]
Then, as before, we may extend $\vec w$ to the entire $\Omega$ by setting $\vec w \equiv 0$ on $\Omega(t+\epsilon^2,\infty)$ so that $\vec w$ belongs to $\rV^{1,0}_2(\Omega)$ and satisfies for all $\tau<t$ the identity
\begin{equation}					\label{eq2.241p}
\int_{\omega(\tau)} w^i \phi^i\,dx +\int_{\Omega(\tau,\infty)} w^i \phi^i_t\,dX +\int_{\Omega(\tau,\infty)} \tilde A^{\alpha\beta}_{ij} D_\beta w^j D_\alpha \phi^i\,dX = \int_{Q_{+}(X,\epsilon)} \Psi_\epsilon \phi^k \,dX
\end{equation}
for all $\vec \phi \in C^\infty_{c}(\Omega)^m$.
We define the averaged Green's function $\tilde{\vec G}{}^\epsilon(\cdot,Y)=(\tilde G{}^\epsilon_{jk}(\cdot,Y))_{j,k=1}^m$ of the adjoint operator $\sLt$ in $\Omega$ by
\[
\tilde G{}^\epsilon_{jk}(\cdot,Y)=w^j=w^j_{\epsilon,Y,k}.
\]
Then by a similar argument, we construct a Green's function $\tilde{\vec G}(\cdot,Y)$ from $\tilde{\vec G}{}^\epsilon(\cdot,Y)$, which belongs to $C^{\mu_0/2,\mu_0}_{loc}(\Omega\setminus\set{Y})$ and satisfies \eqref{eq3.02wb} and the estimates \textit{i) -- vi)}.
Moreover, by following \cite[Lemma~3.5]{CDK}, we obtain the identity \eqref{eq3.01mq}.

Next, we shall prove the identity \eqref{eq5.24}.
Let $\vec \psi_0\in L_2(\omega(s_0))^m$ be given and let $\vec u$ be a unique weak solution of the problem \eqref{eq5.88}.
Fix $X=(t,x)\in\Omega(s_0,\infty)$ and let $\vec w=\vec w_{\epsilon,X,k}$ be as constructed above.
By Lemma~\ref{lem4.1}, for $\epsilon$ sufficiently small, we have
\begin{equation}  \label{eqn:3.56}
\int_{Q_{+}(X,\epsilon)} \Psi_\epsilon  u^k \,dY= \int_{\omega(s_0)} \vec w_{\epsilon,X,k}\cdot \vec \psi_0\,dy.
\end{equation}
For $\vec \psi_0 \in C^\infty_c(\omega(s_0))$, it can be easily seen that (see \cite[Section~3.5]{CDK})
\[
\lim_{\epsilon\to 0}\int_{\omega(s_0)}\vec w_{\epsilon, X, k}  \cdot \vec \psi_0\,dy =\int_{\omega(s_0)} \tilde{\vec G}(\cdot , X)\vec e_k \cdot \vec \psi_0 \,dy,
\]
Since the condition $\IH$ implies that $\vec u$ is continuous at $X$, by taking the limit $\epsilon\to 0$ in \eqref{eqn:3.56} and using \eqref{eq3.01mq}, we obtain
\[
u^k(X)= \int_{\omega(s_0)} \tilde G_{ik}(s_0,y,t,x) \psi_0^i(y)\,dy=\int_{\omega(s_0)} G_{ki}(t,x,s_0,y) \psi_0^i(y)\,dy.
\]
We have thus derived \eqref{eq5.24} under an assumption that $\vec \psi_0\in C^\infty_c(\omega(s_0))$.
For $\vec \psi_0\in L_2(\omega(s_0))^m$, let $\set{\vec \psi_j}_{j=1}^\infty$ be a sequence in $C^\infty_c(\omega(s_0))^m$ such that $\vec \psi_j \to \vec \psi_0$ in $L_2(\omega(s_0))$.
Let $\vec u_j$ be a unique weak solution of the problem \eqref{eq5.88} with $\vec \psi_0$ replaced by $\vec \psi_j$.
Then by Lemma~\ref{lem4.1}, we find that $\lim_{j\to\infty}\norm{\vec u_j-\vec u}_{V_2(\Omega(s_0,t))}=0$ and by the condition $\IH$ and \eqref{eq3.10np} we have $\lim_{j\to\infty}\abs{\vec u_j(X)-\vec u(X)}=0$.
On the other hand, by the estimate \textit{i)} applied to $\tilde{\vec G}(\cdot,X)$ together with the identity \eqref{eq3.01mq}, we find that $\norm{\vec G(t,x,s_0,\cdot)}_{L_2(\omega(s_0))}<\infty$, and thus we get
\[
\lim_{j\to\infty} \int_{\omega(s_0)}\vec G(t,x,s_0,y)\vec \psi_j(y)\,dy=\int_{\omega(s_0)}\vec G(t,x,s_0,y) \vec \psi_0(y)\,dy.
\]
This completes the proof of  \eqref{eq5.24}.
Similarly, for $\vec \psi_1\in L_2(\omega(t_1))^m$, let $\vec u$ be a unique weak solution of the problem
\[
\sLt \vec v=0\;\text{ in }\;\Omega(-\infty,t_1),\quad \vec v=0\;\text{ on }\; S\Omega(-\infty,t_1),\quad \vec v=\vec\psi_1\;\text{ on }\;\omega(t_1).
\]
Then as above, $\vec v$ has the following representation:
\[
\vec v(s,y)=\int_{\omega(t_1)} \tilde{\vec G}(s,y,t_1,x)\vec \psi_1(x)\,dx.
\]

It only remains us to prove \eqref{eq5.23}.
We proceed similar to \cite[Section~4.4]{CDK}.
The following lemma is another simple consequence of Brown et al. \cite{BHL}.
\begin{lemma}					\label{lem4.3}
Let $\eta=\eta(x)\in C^1(\bR^d)$ be a bounded nonnegative function.
Assume that $\vec u\in \rV^{0,1}_2(\Omega(s_0,\infty))$ is the weak solution of the problem \eqref{eq5.88} and define
\[
I(t)=\frac{1}{2} \int_{\omega(t)} \eta(x) \abs{\vec u(t,x)}^2\,dx,\quad t\in (s_0,\infty).
\]
Then $I(t)$ is absolutely continuous and satisfies a.e. $t>s_0$ the identity
\[
I'(t)=-\int_{\omega(t)} A^{\alpha\beta}_{ij}D_\beta u^j D_\alpha(\eta u^i).
\]
\end{lemma}

The following lemmas are key ingredients to prove \eqref{eq5.23} and adapted from \cite{CDK}.
\begin{lemma}					\label{lem4.4}
Assume that $\vec \psi_0\in L_2(\omega(s_0))^m$ is supported in a closed set $F\subset\overline{\omega(s_0)}$ and let $\vec u$ be the weak solution of the problem \eqref{eq5.88}.
Then, we have
\begin{equation}					\label{gaffney}
\int_E \abs{\vec u(t,x)}^2\,dx \leq e^{-\gamma\dist(E,F)^2/(t- s_0)} \int_F \abs{\vec \psi_0(x)}^2\,dx,\quad\forall E\subset \omega(t),
\end{equation}
where $\dist(E,F)=\inf\set{\abs{x-y}\colon\, x \in E,\, y \in F}$ and $\gamma=\gamma(\nu)>0$.
\end{lemma}
\begin{proof}
We may assume that $\dist(E,F)>0$; otherwise \eqref{gaffney} is an immediate consequence of the energy inequality \eqref{eq4.01re}.
Let $\phi=\phi(x)$ be a bounded $C^1$ function on $\bR^d$ satisfying $\abs{D\phi}\leq K$ for some $K>0$ to be fixed later.
Define
\[
I(t)=\int_{\omega(t)} e^{2\phi(x)}\abs{\vec u(t,x)}^2\,dx,\quad t>s_0.
\]
By Lemma~\ref{lem4.3}, we find that $I'(t)$ satisfies for a.e. $t>s_0$
\begin{align*}
I'(t) &=-2\int_{\omega(t)} e^{2\phi}A^{\alpha\beta}_{ij}D_\beta u^j D_\alpha u^i\,dx -4\int_{\omega(t)} e^{2\phi}A^{\alpha\beta}_{ij}D_\beta u^j D_\alpha\phi\,u^i \,dx\\
&\leq -2\nu \int_{\omega(t)} e^{2\phi}\abs{D\vec u}^2\,dx+4(K/\nu) \int_{\omega(t)} e^\phi\abs{D\vec u} e^\phi\abs{\vec u}\,dx\\
&\leq(2/\nu^3) K^2 \int_{\omega(t)} e^{2\phi}\abs{\vec u}^2 =(2/\nu^3) K^2 I(t).
\end{align*}
The above differential inequality yields
\begin{equation} 					\label{eqx001}
I(t) \leq e^{(2/\nu^3) K^2(t-s_0)} \norm{e^\phi \vec \psi_0}_{L_2(F)}^2,\quad \forall t\geq s_0.
\end{equation}
Notice that by a standard approximation, we may assume that $\phi$ is a bounded Lipschitz continuous function satisfying $\abs{D \phi}\leq K$ a.e.
Since $F$ is a closed set, the function
\[
\dist(x, F)=\inf\set{\abs{x-y}\colon\, y \in F}
\]
is a Lipschitz function on $\bR^d$ with Lipschitz constant $1$ and $\dist(E,F)=\inf_{x\in E} \dist(x,F)$.
Therefore, if we set $\phi(x)=K(\dist(x,F)\wedge \dist(E,F))$, then by \eqref{eqx001}, we get
\[
\int_E \abs{\vec u(t,x)}^2\,dx \leq \exp\Set{(2/\nu^3) K^2 (t-s_0)-2K\dist(E,F)} \int_F \abs{\vec \psi_0(x)}^2\,dx.
\]
The lemma follows if we set $K=\dist(E,F)/\set{(2/\nu^3)(t-s_0)}$.
\end{proof}

\begin{lemma}  \label{lem4.5}
Let $\vec u$ be the weak solution of the problem \eqref{eq5.88}, where $\vec \psi_0 \in L_\infty(\omega(s_0))$ and has a compact support in $\omega(s_0)$.
Denote
\begin{equation}					\label{eq4.27hs}
\varrho=\varrho(x)=\dist(x,\partial\omega(s_0))\wedge R_c,\quad x\in \omega(s_0)
\end{equation}
Then for all $x\in\omega(s_0)$, we have
\[
\abs{\vec u(t,x)}\leq N \norm{\vec\psi_0}_{L_\infty(\omega(s_0))}\;\text{ whenever }\;0<t-s_0<(1\wedge M^{-2})\varrho^2(x)/4,
\]
where $N=N(d,m,\nu, \mu_0,C_0)>0$.
\end{lemma}
\begin{proof}
For $x\in \omega(s_0)$, set $r=\varrho(x)/2$ and $\delta=(r/M)^2$ so that
\[
(s_0-\delta, s_0+\delta)\times B(x,r) \subset\subset \Omega.
\]
For any $t$ satisfying $0<t-s_0<\delta\wedge r^2=(1\wedge M^{-2})\varrho^2(x)/4$, set $R=\sqrt{t-s_0}\,$ and denote
\[
A_0=B(x,R);\quad A_k=\set{y\in\omega(s_0)\colon\, 2^{k-1}R \leq \abs{y-x}<2^k R},\quad k=1,2,\ldots.
\]
Since $\vec \psi_0$ is compactly supported in $\omega(s_0)$, we have $\vec \psi_0=\sum_{k=0}^{k_0}\chi_{A_k} \vec \psi_0 $ for some $k_0<\infty$.
For $k=0,1,\ldots, k_0$, we define
\[
\vec u_k(t,x)=\int_{A_k} \vec G(t,x,s_0,y)\vec \psi_0(y)\,dy.
\]
Then, it follows from \eqref{eq5.24} that $\vec u=\sum_{k=0}^{k_0} \vec u_k$ and that each $\vec{u}_k$ is the weak solution of the problem \eqref{eq5.88} with $\chi_{A_k} \vec \psi_0$ in place of $\vec \psi_0$.
We apply Lemma~\ref{lem4.1} to $\vec u_k$ with $E=B(x,R)$ and $F=\overline A_k$ for $k=1,2,\ldots$, to obtain that 
\[
\int_{B(x,R)}\abs{\vec u_k(s,y)}^2\,dy \leq N e^{-\gamma(2^{k-1}-1)}2^{kd} R^d\norm{\vec \psi_0}_{L_\infty(\omega(s_0))}^2, \quad\forall s\in (s_0,t).
\]
Therefore, by the condition $\IH$ and \eqref{eq3.10np}, we get
\[
\abs{\vec u(t,x)} \leq \sum_{k=0}^{k_0} \abs{\vec u_k(t,x)} \leq N \left(1+\sum_{k=1}^\infty e^{-\gamma (2^{k-1}-1)}2^{kd/2}\right) \norm{\vec \psi_0}_{L_\infty(\omega(s_0))} \leq N \norm{\vec \psi_0}_{L_\infty(\omega(s_0))}.
\]
The lemma is proved.
\end{proof}

\begin{lemma}					\label{lem4.6}
Let $\eta\in C^\infty_c(B(x_0,4r))$ be a function satisfying
\begin{equation}					\label{eq4.28th}
0\leq \eta \leq 1,\quad \eta\equiv 1\;\text{ in }\; B(x_0,2r),\quad \text{and }\;\abs{D\eta}\leq 4/r,
\end{equation}
where $x_0\in\omega(s_0)$ and $r<\varrho(x_0)/5$, where $\varrho$ is as defined in \eqref{eq4.27hs}.
Then, we have
\[
\lim_{t \to s_0} \int_{\omega(s_0)} \vec G(t,x,s_0,y) \eta(y)\,dy =  I_m,\quad \forall x\in B(x_0,r).
\]
\end{lemma}
\begin{proof}
By taking $\epsilon\to 0$ in \eqref{eq2.241p} and arguing as in the proof of \cite[Lemma~4.3]{CDK}, we find that the following identity holds for all $\tau<t$:
\begin{multline}					\label{eq3.431p}
\phi^k(X)=\phi^k(t,x)=\int_{\omega(\tau)} \tilde G_{ik}(\tau,y,t,x)\phi^i(\tau,y)\,dy+\int_{\Omega(\tau,t)} \tilde G_{ik}(Y,X) \phi^i_s(Y)\,dY\\
+\int_{\Omega(\tau,t)}\tilde A^{\alpha\beta}_{ij} D_{y^\beta} \tilde G_{jk}(Y,X)D_\alpha\phi^i(Y)\,dY,\quad \forall \vec \phi \in C^\infty_c(\Omega)^m,
\end{multline}
where we have used \eqref{eq3.02wb}.
Let $\zeta=\zeta(s)$ be a smooth function on $\bR$ such that
\[
0\leq \zeta \leq 1,\quad \zeta(s)=1\;\text{ for }\;\abs{s-s_0}\leq \delta,\quad\text{and }\;\zeta(s)=0\;\text{ for }\;\abs{s-s_0}\geq 2\delta,
\]
where $\delta$ is chosen so small that
\[
(s_0-2\delta, s_0+2\delta)\times B(x_0,4r) \subset\subset \Omega.
\]
Notice that we may take $\vec \phi(Y)=\vec\phi(s,y)=\zeta(s)\eta(y)\vec e_l$  in \eqref{eq3.431p}.
Setting $\tau=s_0$ and assuming that $\abs{t-s_0}<\delta$  in \eqref{eq3.431p}, we obtain by \eqref{eq3.01mq} that
\begin{equation}					\label{eq387p}
\delta_{kl}=\int_{\omega(s_0)}G_{kl}(t,x,s_0,y) \eta(y)\,dy+\int_{\Omega(s_0,t)} \tilde A^{\alpha\beta}_{lj} D_{y^\beta} \tilde G_{jk}(Y,X) D_\alpha\eta(y) \,dY=:I+II.
\end{equation}
Then for all $X=(t,x)$ such that $x\in B(x_0,r)$ and $\abs{t-s_0}<\delta \wedge r^2$, we estimate $II$ as follows by using the hypothesis \eqref{eq4.28th}, H\"older's inequality, and the estimate \textit{i)} for $\tilde{\vec G}(\cdot, X)$:
\[
\abs{II} \leq N r^{d/2-1}(t-s_0)^{1/2} \left(\int_{\Omega(s_0,t)\setminus \overline Q(X,r)} \abs{D_{y} \tilde{\vec G}(Y,X)}^2\,dY\right)^{1/2} \leq C r^{-1}(t-s_0)^{1/2}.
\]
Therefore, the lemma follows by taking the limit $t$ to $s_0$ in \eqref{eq387p}.
\end{proof}
We are ready to prove \eqref{eq5.23}.
Let $\vec \psi_0\in L_2(\omega(s_0))^m$ and assume that $\vec \psi_0$ is continuous at $x_0\in \omega(s_0)$.
Let $\vec u$ be the weak solution of the problem \eqref{eq5.88}.
For any $\epsilon>0$ given, choose $r<\varrho(x_0)/5$, where $\varrho$ is as defined in \eqref{eq4.27hs}, such that
\[
\abs{\vec \psi_0(x)-\vec \psi_0(x_0)}< \epsilon/2N\;\text{ for all $x$  satisfying $\abs{x-x_0}<4r$},
\]
where $N$ is the constant that appears in Lemma~\ref{lem4.5}.
Let $\eta$ be given as in Lemma~\ref{lem4.6} and let $\vec u_0$, $\vec u_\epsilon$, and $\vec u_\infty$, respectively, be the weak solution of the problem \eqref{eq5.88} with $\eta\vec \psi_0(x_0)$, $\eta(\vec \psi_0-\vec \psi_0(x_0))$, and $(1-\eta)\vec \psi_0$ in place of $\vec \psi_0$.
By the uniqueness, we have $\vec{u}=\vec{u}_0+\vec{u}_\epsilon+\vec{u}_\infty$ and by the formula \eqref{eq5.24}, $\vec u_0$ is represented by
\begin{equation}  \label{eq002y}
\vec u_0(t,x)=\left(\int_{\omega(s_0)} \vec G(t,x,s_0,y) \eta(y)\,dy\right) \vec \psi_0(x_0).
\end{equation}
Let $\delta>0$ be chosen so that
\[
(s_0-\delta, s_0+\delta)\times B(x_0,4r) \subset\subset \Omega.
\]
For any $s$ satisfying $0<s-s_0<\delta$, we set $E=B(x_0,r)\subset \omega(s)$ and $F=\overline{\omega(s_0)}\setminus B(x_0,2r)$ in Lemma~\ref{lem4.4} to get
\[
\int_{B(x_0,r)}\abs{\vec u_\infty(s,y)}^2\,dy \leq e^{-\gamma r^2/(s-s_0)} \norm{\vec \psi_0}_{L_2(\omega(s_0))}^2.
\]
Therefore, for all $t$ satisfying $0<t-s_0<\delta \wedge r^2$, we set $R=\sqrt{t-s_0}/4$ in \eqref{eq3.10np} to get
\begin{equation}					\label{eq001y}
\abs{\vec u_\infty(t,x)} \leq N R^{-d/2} e^{-\gamma(r/R)^2} \norm{\vec \psi_0}_{L_2(\omega(s_0))},\quad\forall x\in B(x_0,R).
\end{equation}
Finally, we estimate $\vec u_\epsilon$ by using Lemma~\ref{lem4.5}.
\begin{equation}					\label{eq003y}
\abs{\vec u_\epsilon(t,x)}\leq  \epsilon/2\;\text{ whenever }\; 0<t-s_0<(1\wedge M^{-2})\,\varrho^2(x)/4.
\end{equation}
Combining \eqref{eq002y}, \eqref{eq001y}, and \eqref{eq003y}, we see that  if $t-s_0$ is chosen sufficiently small, then there exists $\vartheta>0$ such that for all $x\in B(x_0,\vartheta)$ we have $\abs{\vec u(t,x)-\vec\psi_0(x_0)}<\epsilon$.
This completes the proof.
\hfill\qedsymbol


\subsection{Proof of Theorem~\ref{thm2}}
By Theorem~\ref{thm1}, the condition $\IH$ implies existence of the Green's function $\vec G(X,Y)$ of $\sL$ in $\Omega$.
Therefore, all the conclusions of Theorem~\ref{thm1} are satisfied.
Let $\phi$ be a bounded Lipschitz function on $\bR^d$ satisfying $\abs{D \phi} \leq K$ a.e. for some $K>0$ to be chosen later.
For any $\vec f\in L_2(\omega(s))^m$, let $\vec u$ be a unique weak solution of the problem
\begin{equation} 		\label{eq3.6.1}
\sL \vec u = 0\;\text{ in }\;\Omega(s,\infty),\quad \vec u=0\;\text{ on }\; S\Omega(s,\infty),\quad \vec u= e^{-\phi}\vec f\;\text{ on }\;\omega(s).
\end{equation}
For $t>s$, we define the operator $P^\phi_{s\to t}: L_2(\omega(s))^m\to L_2(\omega(t))^m$ by
\[
P^\phi_{s\to t} \vec f(x)= e^{\phi(x)}\vec u(t,x).
\]
Notice that by the representation formula \eqref{eq5.24}, we have
\begin{equation}					\label{eq3.60.3}
P^\phi_{s\to t}\vec f(x)= e^{\phi(x)}\int_{\omega(s)}\vec G(t,x,s,y)e^{-\phi(y)}\vec f(y)\,dy.
\end{equation}
For $t\geq s$, we define
\[
I(t)=\norm{e^\phi\vec u(t,\cdot)}_{L^2(\omega(t))}^2=\norm{P^\phi_{s\to t}\vec f}_{L_2(\omega(t))}^2.
\]
Then, as in the proof of Lemma~\ref{lem4.4},
we find that $I$ is absolutely continuous and satisfies for a.e. $t>s$, the differential inequality
\[
I'(t) \leq (2/\nu^3) K^2 I(t).
\]
The above inequality with the initial condition $I(s)=\norm{\vec f}^2_{L_2(\omega(s))}$ yields
\[
I(t)\leq e^{(2/\nu^3) K^2 (t-s)}\norm{\vec f}_{L_2(\omega(s))}^2, \quad \forall t \geq s.
\]
We have thus shown that for all $\vec f \in L_2(\omega(s))^m$, the operator $P^\phi_{s\to t}$ satisfies
\begin{equation}					\label{eq3.70}
\norm{P^\phi_{s\to t}\vec f}_{L_2(\omega(t))} \leq e^{\vartheta K^2(t-s)}\norm{\vec f}_{L_2(\omega(s))}, \quad \forall t>s,
\end{equation}
where $\vartheta=\nu^{-3}$.
We set $R= \sqrt{t-s} \wedge R_{max}$ and use the condition $\LB$ to estimate
\begin{align*}
e^{-2\phi(x)}\abs{P^\phi_{s\to t}\vec f(x)}^2 = \abs{\vec u(t,x)}^2 &= \abs{\vec u(X)}^2\\
&\leq N_0^2 R^{-(d+2)} \int_{\Omega_{-}[X,R]} \abs{\vec u(Y)}^2\,dY\\
& = N_0^2 R^{-(d+2)} \int_{t-R^2}^t \int_{\omega(\tau)\cap B(x,R)}e^{-2\phi(y)} \abs{P^\phi_{s\to\tau}\vec f(y)}^2\,dy\,d\tau,
\end{align*}
Therefore, by using the estimate \eqref{eq3.70}, we get
\begin{align*}
\abs{P^\phi_{s\to t}\vec f(x)}^2 &\leq N_0^2 R^{-d-2} \int_{t-R^2}^t \int_{\omega(\tau)\cap B(x,R)}e^{2\phi(x)-2\phi(y)} \abs{P^\phi_{s\to\tau}\vec f(y)}^2\,dy\,d\tau\\
& \leq N_0^2 R^{-d-2} \int_{t-R^2}^t \int_{\omega(\tau)\cap B(x,R)}e^{2KR} \abs{P^\phi_{s\to\tau}\vec f(y)}^2\,dy\,d\tau\\
& \le N_0^2 R^{-d-2} \, e^{2KR} \int_{t-R^2}^t e^{2\vartheta K^2(\tau-s)}\norm{\vec f}_{L_2(\omega(s))}^2\,d\tau\\
& \le N_0^2 R^{-d}\,e^{2KR+2\vartheta K^2(t-s)}\norm{\vec f}_{L_2(\omega(s))}^2.
\end{align*}
We have thus obtained the following $L_2(\omega(s)) \to L_\infty(\omega(t))$ estimate for $P^\phi_{s\to t}$:
\begin{equation} \label{eq3.70.3}
\norm{P^\phi_{s\to t}\vec f}_{L_\infty(\omega(t))} \leq  N_0 R^{-d/2}\,e^{KR+\vartheta K^2(t-s)} \norm{\vec f}_{L_2(\omega(s))}.
\end{equation}

We also define the operator $Q^\phi_{t\to s}: L_2(\omega(t))^m\to L_2(\omega(s))^m$ for $s<t$ by setting
\[
 Q^\phi_{t\to s}\vec g(y)= e^{-\phi(y)}\vec v(s,y),\quad  \forall \vec g \in L_2(\omega(t))^m,
 \]
 where $\vec v$ is a unique weak solution of the problem
\begin{equation}					\label{eq3.6.11}
\sLt \vec v = 0\;\text{ in }\;\Omega(-\infty,t),\quad \vec v=0\;\text{ on }\; S\Omega(-\infty,t),\quad \vec v= e^{\phi}\vec g \;\text{ on }\;\omega(t).
\end{equation}
By a similar calculation that leads to \eqref{eq3.70.3}, we obtain
\begin{equation}					\label{eq3.73.3}
\norm{Q^\phi_{t\to s}\vec g}_{L_\infty(\omega(s))} \leq  N_0 R^{-d/2}\,e^{KR+\vartheta K^2(t-s)} \norm{\vec g}_{L_2(\omega(t))}.
\end{equation}
It follows from \eqref{eq3.6.1}, \eqref{eq3.6.11}, and \eqref{eq4.02un} in Lemma~\ref{lem4.1} that
\[
\int_{\omega(t)}\bigl(P^\phi_{s\to t}\vec f\bigr) \cdot \vec g \,dx= \int_{\omega(s)}\vec f\cdot \bigl(Q^\phi_{t\to s} \vec g\bigr)\,dx,\quad \forall \vec f\in L_2(\omega(s))^m,\;\; \forall \vec g \in L_2(\omega(t))^m.
\]
In particular, the above identity holds for all $\vec f \in C^\infty_c(\omega(s))^m$ and $\vec g \in C^\infty_c(\omega(t))^m$.
Therefore, by the estimate \eqref{eq3.73.3} and duality, we get
\begin{equation}					\label{eq3.74.3}
\norm{P^\phi_{s\to t}\vec f}_{L_2(\omega(t))} \leq  N_0 R^{-d/2}\,e^{KR+\vartheta K^2(t-s)} \norm{\vec f}_{L_1(\omega(s))},\quad \forall \vec f \in C^\infty_c(\omega(s))^m.
\end{equation}
Now, set $r=(s+t)/2$ and observe that by uniqueness, we have
\[
P^\phi_{s\to t}\vec f= P^\phi_{r\to t}(P^\phi_{s\to r}\vec f),\quad \forall \vec f\in C^\infty_c(\omega(s))^m.
\]
Then, by noting that $t-r=r-s=(t-s)/2$ and $R/\sqrt{2}\le \sqrt{t-r}\wedge R_{\max}\leq R$, we obtain from \eqref{eq3.70.3} and \eqref{eq3.74.3} that
\[
\norm{P^\phi_{s\to t}\vec f}_{L_\infty(\omega(t))} \leq N R^{-d}\,e^{ 2KR+\vartheta K^2(t-s)} \norm{\vec f}_{L_1(\omega(s))}, \;\; \forall \vec f\in C^\infty_c(\omega(s))^m;\quad N=2^{d/2} N_0^2.
\]
For all $x\in \omega(t)$ and $y\in \omega(s)$, the above estimate and \eqref{eq3.60.3} yield, by duality, that
\begin{equation}					\label{eq3.83}
e^{\phi(x)-\phi(y)}\abs{\vec G(t,x,s,y)} \leq  N R^{-d}\,e^{ 2KR+\vartheta K^2(t-s)}.
\end{equation}
Let $\phi(z)=K \phi_0(\abs{z-y})$, where $\phi_0$ is defined on $[0,\infty)$ by
\[
\phi_0(r)=\begin{cases}r&\text{if $r \leq \abs{x-y}$} \\
\abs{x-y}&\text{if  $r>\abs{x-y}$}.
\end{cases}
\]
Then, $\phi$ is a bounded Lipschitz function on $\bR^d$ satisfying $\abs{D\phi} \leq K$ a.e.
We set
\[
K=\abs{x-y}/2\vartheta(t-s)\quad\text{and}\quad \xi:=\abs{x-y}/\sqrt{t-s}.
\]
By \eqref{eq3.83} and the obvious inequality $R/\sqrt{t-s}\leq 1$, we have
\[
\abs{\vec G (t,x,s,y)}\leq  N R^{-d}\, \exp\{\xi/\vartheta-\xi^2/4\vartheta\}.
\]
Let $N=N(\vartheta)=N(\nu)$ be chosen so that
\[
\exp(\xi/\vartheta-\xi^2/4\vartheta)\le N\exp(-\xi^2/8\vartheta),\quad\forall \xi\in [0,\infty).
\]
If we set $\kappa=1/8\vartheta=\nu^3/8$, then we obtain
\[
\abs{\vec G(t,x,s,y)}\le NR^{-d}\exp\left\{-\kappa\abs{x-y}^2/(t-s)\right\}
\]
where $N=N(d,m,\nu,N_0)>0$ and recall that we set $R=\sqrt{t-s}\wedge R_{max}$.
The proof is complete.
\hfill\qedsymbol


\subsection{Proof of Theorem~\ref{thm3}}
Notice that by Lemma~\ref{lem2.19} and Theorem~\ref{thm2}, for all $X=(t,x)$ and $Y=(y,s) $ in $\Omega$ with $t>s$ we have
\begin{equation}			\label{eq2.25t}
\abs{\vec G(t,x,s,y)}\leq C_1 \left\{(t-s)\wedge R_{max}^2\right\}^{-d/2}\exp\left\{-\kappa\abs{x-y}^2/(t-s)\right\},
\end{equation}
where $C_1=C_1(n,m,\nu, \mu_0,N_1)$.
We denote
\[
\delta_1(X,Y)= \left(1 \wedge \frac {d^{-}(X)} {R_{max}\wedge \abs{X-Y}_\sP} \right)\quad\text{and}\quad \delta_2(X,Y)= \left(1 \wedge \frac{d^{+}(Y)} {R_{max}\wedge \abs{X-Y}_\sP}\right)
\]
so that $\delta(X,Y)=\delta_1(X,Y)\, \delta_2(X,Y)$.
To prove the estimate \eqref{eq3.8yy}, we first claim that
\begin{equation}			\label{eq22.00k}
\abs{\vec G(t,x,s,y)}\leq N \delta_1(X,Y)^{\mu_0} \left\{(t-s)\wedge R_{max}^2\right\}^{-d/2}\exp\left\{-\kappa \abs{x-y}^2/4(t-s)\right\},
\end{equation}
where $N=N(n,m,\nu, \mu_0,N_1)$.
The following lemma is a key to prove the above claim.

\begin{lemma}			\label{lem3.6}
For $R \in (0,R_{max})$ and $X\in\Omega$ such that $d^{-}(X)< R/2$, let $\vec u$ be a weak solution of $\sL\vec u=0$ in $\Omega_{-}[X,R]$ vanishing on $\sP \Omega_{-}[X,R]$.
Then, we have
\begin{equation}			\label{eq3.7m}
\abs{\vec u(X)} \le N d^{-}(X)^{\mu_0} R^{-d/2-1-\mu_0}\norm{\vec u}_{L_2(\Omega_{-}[X,R])},
\end{equation}
where $N=N(d,m,\nu, \mu_0, N_1,M)$.
\end{lemma}

\begin{proof}
By the very definition the condition $\LH$, we have
\begin{equation}			\label{eq3.8x}
\abs{\tilde{\vec u}(Y)-\tilde{\vec u}(X)} \leq N \abs{Y-X}_\sP^{\mu_0}\, R^{-d/2-1-\mu_0}\norm{\vec u}_{L_2(\Omega_{-}[X,R])},\quad\forall Y\in Q_{-}(X,R/2).
\end{equation}
For any $r$ satisfying $d^{-}(X)<r< R/2$, there is $Y \in Q_{-}(X,R/2) \setminus \Omega$ such that $\abs{X-Y}_\sP=r$.
By \eqref{eq3.8x} we obtain
\[
\abs{\vec u(X)}=\abs{\tilde{\vec u}(X)-\tilde{\vec u}(Y)} \leq N r^{\mu_0} R^{-d/2-1-\mu_0}\norm{\vec u}_{L_2(\Omega_{-}[X,R])}.
\]
By taking limit $r\to d^{-}(X)$ in the above inequality, we derive \eqref{eq3.7m}.
\end{proof}

Now we are ready to prove \eqref{eq22.00k}.
Take $R=(R_{max}\wedge \abs{X-Y}_\sP)/4$.
We may assume that $d^{-}(X) < R/2$ because otherwise \eqref{eq22.00k} follows from \eqref{eq2.25t}.
We then set $\vec u$ to be the columns of $\vec G(\cdot,Y)$ in Lemma~\ref{lem3.6} to obtain
\begin{equation}							\label{eq17.24}
\abs{\vec G(X,Y)} \leq C d^{-}(X)^{\mu_0} R^{-d/2-1-\mu_0} \norm{\vec G(\cdot,Y)}_{L_2(\Omega_{-}[X,R])},\quad R=(R_{max}\wedge \abs{X-Y}_\sP)/4.
\end{equation}

Next,  we consider the following three possible cases.

\begin{case1}
In this case $R=\sqrt{t-s}/4=\abs{X-Y}_\sP/4$ and thus, we get from \eqref{eq17.24} and \eqref{eq5.18} that
\[
\abs{\vec G(X,Y)} \leq N d^{-}(X)^{\mu_0} R^{-d/2-1-\mu_0} \norm{\vec G(\cdot,Y)}_{L_2(\Omega_{-}[X,R])}\leq N d^{-}(X)^{\mu_0} R^{-d-\mu_0},
\]
which immediately implies \eqref{eq22.00k} in this case.
\end{case1}

\begin{case2}
In this case $R=(\abs{x-y}\wedge R_{max})/4$.
We denote $Z=(r,z)$ and claim that for all $Z \in \Omega_{-}(X,2R)$, we have
\begin{equation}			\label{eq5.26rv}
\abs{\vec G(r,z,s,y)} \leq N C_1(t-s)^{-d/2} \exp\left\{-\kappa \abs{x-y}^2/4(t-s)\right\},
\end{equation}
where $C_1$ and $\kappa$ are the same constants as in \eqref{eq2.25t} and $N=N(d,\kappa)$.
To prove the claim, first note that we may assume $Y=0$ without loss of generality.
Then by \eqref{eq2.25t} we have
\[
\abs{\vec G(r,z,s,y)} \leq C_1 r^{-d/2} e^{-\kappa \abs{z}^2/r} \, \chi_{(0,\infty)}(r)  \leq C_1 r^{-d/2} e^{-\kappa \abs{x}^2/4r} \, \chi_{(0,\infty)}(r),
\]
where we used $\abs{z}=\abs{z-y} \geq \abs{x-y}/2=\abs{x}/2$.
Let us denote
\[
g(\tau)= \tau^{-d/2} e^{-\kappa \abs{x}^2/ 4\tau}\,\chi_{(0,\infty)}(\tau),\quad g_0(\tau)=\tau^{-d/2} e^{-\kappa/4 \tau}\,\chi_{(0,\infty)}(\tau).
\]
Then the claim \eqref{eq5.26rv} will follow if we show that there exists a positive number $N=N(d,\kappa)$ such that $g(r)<N g(t)$ for all $r <t < \abs{x}^2$, which in turn will follow if  we show that $g_0(r_1) \leq N g_0( r_2)$ for all $r_1<r_2 \leq 1$.
But the latter assertion is easy to verify by an elementary analysis of the function $g_0$.

We have thus proved \eqref{eq5.26rv}, which combined with \eqref{eq17.24} yields
\[
\abs{\vec G(X,Y)} \leq C d^{-}(X)^{\mu_0} R^{-\mu_0} (t-s)^{-d/2} \exp \left\{-\kappa \abs{x-y}^2/4(t-s)\right\}.
\]
Therefore, we also obtain \eqref{eq22.00k} in this case.
\end{case2}

\begin{case3}
In this case $R=R_{max}/4$, and the desired estimate \eqref{eq22.00k} becomes
\begin{equation}			\label{eq22.15}
\abs{\vec G(t,x,s,y)}\leq C\set{d^{-}(X)/R_{max}}^{\mu_0} R_{max}^{-d}\exp\bigset{-\kappa\abs{x-y}^2/4(t-s)}.
\end{equation}
Since $t-s\geq 16 R^2$, for all $Z=(r,z)\in \Omega_{-}(X,2R)$, we have
\begin{equation}
                        \label{eq18.21}
\exp\left\{- \kappa\, \frac{\abs{z-y}^2}{r-s}\right\}\le
\exp\left\{- \kappa\, \frac{\abs{x-y}^2/2-\abs{z-x}^2}{t-s}\right\} \leq e^{\kappa/4} \exp\left\{- \frac{\kappa\abs{x-y}^2}{2(t-s)}\right\}.
\end{equation}
Then, from \eqref{eq17.24}, \eqref{eq2.25t}, and \eqref{eq18.21}, we obtain \eqref{eq22.15}, which implies \eqref{eq22.00k} in this case.
\end{case3}

We have thus proved that the estimate \eqref{eq22.00k} holds in all possible cases.
Finally, notice that Lemma~\ref{lem3.6} remains valid if $\sL$, $X$, $d^{-}(X)$, $\Omega_{-}[X,R]$, and $\sP\Omega_{-}[X,R]$, respectively, are replaced by $\sLt$, $Y$, $d^{+}(Y)$, $\Omega_{+}[Y,R]$, and $\sP\Omega_{+}[Y,R]$.
Therefore, by replicating the above argument to $\tilde{\vec G}(\cdot,X)$, utilizing the estimate \eqref{eq22.00k} instead of \eqref{eq2.25t}, and using the identity \eqref{eq3.01mq}, we obtain 
\[
\abs{\vec G(t,x,s,y)}\leq N \delta_2(X,Y)^{\mu_0}\delta_1(X,Y)^{\mu_0} \{(t-s)\wedge R_{max}^2\}^{-d/2}\exp\left\{-\kappa \abs{x-y}^2/16(t-s) \right\}.
\]
By replacing $\kappa$ by $\kappa/16$, we obtain the desired estimate \eqref{eq3.8yy}.
The theorem is proved.
\hfill\qedsymbol

\section{Appendix}

\begin{lemma}				\label{lem2.19}
Assume the condition $\LH$.
Then the condition $\LB$ is satisfied with the same $R_{max}$ and $N_0=N_0(n,m,\mu_0,N_1)$.
\end{lemma}

\begin{proof}
Let $\vec u$ be a weak solution of $\sL \vec u=0$ in $\Omega_{-}[X,R]$ vanishing on $\sP\Omega_{-}[X,R]$, where $X\in \Omega$ and $R\in (0,R_{max})$.
By using the triangle inequality, for all $Y\in Q_{-}(X,R/2)$ and $Z\in Q_{-}(Y,R/2)$, we have
\[
\abs{\tilde{\vec u}(Y)}^2 \leq 2\abs{\tilde{\vec u}(Y)-\tilde{\vec u}(Z)}^2+2\abs{\tilde{\vec u}(Z)}^2 \leq N R^{2\mu_0} [\tilde{\vec u}]_{\mu_0/2,\mu_0;Q_{-}(X,R)}^2+N\abs{\tilde{\vec u}(Z)}^2.
\]
Then by taking average over $Z\in Q_{-}(Y,R/2)$ and using $\LH$, we obtain
\[
\norm{\vec u}_{L_\infty(\Omega_{-}[X,R/2])}^2 \leq N R^{2\mu_0} [\tilde{\vec u}]_{\mu_0/2,\mu_0;Q_{-}(X,R)}^2+N R^{-d-2}\norm{\tilde{\vec u}}_{L_2(Q_{-}(X,R))}^2 \leq N R^{-d-2}\norm{\vec u}_{L_2(\Omega_{-}[X,R])}^2,
\]
where $N=N(d,m,\mu_0, N_1)$.
This proves the part i) of the condition $\LB$.
The proof for the other part is very similar and is omitted.
\end{proof}

\begin{lemma}					\label{lem8.1ap}
Let $\Omega$ be a time-varying $H_1$ (graph) domain in $\bR^{d+1}$.
Assume that $\vec u$ is a weak solution of $\sL \vec u=\vec f$ in $\Omega_{-}[X_0,R]$ vanishing on $\sP\Omega_{-}[X_0,R]$, where $\vec f \in L_\infty(\Omega_{-}[X_0,R])$, and denote $\tilde{\vec u}=\chi_{\Omega_{-}[X_0,R]}\, \vec u$.
Then we have
\begin{equation}		\label{eq8.6ap}
\int_{Q_{-}(X_0,R)} \bigabs{\tilde{\vec u}-(\tilde{\vec u})_{X_0,R}}^2 \leq N R^2\int_{\Omega_{-}[X_0,R]}\abs{D\vec u}^2 +N R^{2-d}\norm{\vec f}_{L_1(\Omega_{-}[X_0,R])}^2,
\end{equation}
where $(\tilde{\vec u})_{X_0,R}=\fint_{Q_{-}(X_0,R)} \tilde{\vec u}$ and $N=N(d,m,\nu)$.
\end{lemma}

\begin{proof}
We modify the proof of \cite[Lemma~3]{Struwe}.
Without loss of generality, we may assume $X_0=0$.
Let $\zeta=\zeta(x)$ be a smooth function defined on $\bR^d$ such that
\[
0\leq \zeta \leq 1,\quad \supp \zeta \subset B(R),\quad \zeta\equiv 1\,\text{ on }\, B(R/2),\quad\text{and}\quad  \abs{D\zeta} \le 4/R.
\]
Setting $\delta^{-1}=\int_{B(R)}\zeta(x)\,dx$ and denote
\[
\vec\beta(t):= \delta \int_{B(R)} \zeta(x) \tilde{\vec u}(t,x)\,dx=\delta \int_{\omega(t)\cap B(R)}\zeta(x) \vec u(t,x)\,dx,\quad \bar{\vec \beta}:=R^{-2}\int_{-R^2}^0 \vec \beta(t)\,dt.
\]
Since $(\tilde{\vec u})_{0,R}$ minimizes the integral $\int_{Q_{-}(R)} \abs{\tilde{\vec u}-\vec c}^2$ among $\vec c\in \bR^m$, we obtain
\begin{equation}					\label{eq6.10gs}
\int_{Q_{-}(R)} \abs{\tilde{\vec u}-(\tilde{\vec u})_{0,R}}^2  \leq \int_{Q_{-}(R)} \abs{\tilde{\vec u}-\bar{\vec \beta}}^2 \leq 2 \int_{Q_{-}(R)} \abs{\tilde{\vec u}-\vec\beta(t)}^2+ 2 \int_{Q_{-}(R)} \abs{\vec \beta(t)-\bar{\vec \beta}}^2.
\end{equation}
Notice that $\tilde{\vec u}\in W^{0,1}_2(Q_{-}(R))$ and $D \tilde{\vec u}= \chi_{\Omega_{-}[R]}\, D \vec u$ in $Q_{-}(R)$.
Therefore, by a variant of Poincar\'e's inequality, we have
\begin{equation}					\label{eq6.11sj}
 \int_{Q_{-}(R)}\abs{\tilde{\vec u}-\vec\beta(t)}^2 \,dX=\int_{-R^2}^0 \int_{B(R)} \abs{\tilde{\vec u}(t,x)-\vec\beta(t)}^2\,dx\,dt \leq N \int_{\Omega_{-}[R]} \abs{D\vec u}^2\,dX.
 \end{equation}
We claim that for all $s$ and $t$ satisfying $-R^2 <s<t<0$, we have
\begin{equation}				\label{eq8.8ap}
\abs{\vec \beta(t)-\vec\beta(s)}^2 \leq N R^{-d} \int_{\Omega_{-}[R]} \abs{D\vec u}^2 + NR^{-2d}\norm{\vec f}^2_{L_1(\Omega_{-}[R])}.
\end{equation}
Assume the estimate \eqref{eq8.8ap} for the moment.
By the definition of $\bar{\vec \beta}$, we then obtain
\begin{multline}					\label{eq6.14ch}
\int_{Q_{-}(R)} \abs{\vec \beta(t)-\bar{\vec \beta}}^2\,dX =  \abs{B(R)} \int_{-R^2}^0 \abs{\vec \beta(t)-\bar{\vec \beta}}^2 \,dt\\
\leq  NR^{d-2} \int_{-R^2}^0 \!\int_{-R^2}^0 \bigabs{\vec\beta(t)-\vec\beta(s)}^2\,ds\,dt \leq N R^2 \int_{\Omega_{-}[R]} \abs{D\vec u}^2\,dX+ NR^{2-d}\norm{\vec f}^2_{L_1(\Omega_{-}[R])}.
\end{multline}
By combining \eqref{eq6.10gs}, \eqref{eq6.11sj}, and \eqref{eq6.14ch}, we obtain \eqref{eq8.6ap}.

It remains us to prove the estimate \eqref{eq8.8ap}.
Setting $\vec \eta=\vec \eta(x)=\zeta(x) \vec \Lambda$, where $\vec \Lambda \in \bR^m$ is a constant column vector, and following the calculation in Brown et al. \cite{BHL}, we obtain
\begin{multline*}
\int_{\omega(t)\cap B(R)} \zeta(x) \vec u(t,x) \cdot \vec \Lambda\,dx-\int_{\omega(s)\cap B(R)} \zeta(x) \vec u(s,x) \cdot \vec \Lambda\,dx\\
+\int_{\Omega(s,t)} \vec \Lambda^T \vec A^{\alpha\beta}(X) D_\beta \vec u(X) D_\alpha \zeta(x)\,dX=\int_{\Omega(s,t)}\vec f(X) \cdot \vec \Lambda\zeta(x)\,dX.
\end{multline*}
Notice that $\delta^{-1} \geq  2^{-d}\abs{B(1)} R^d$.
Therefore, by using the properties of  the function $\zeta$, we get for all $s$ and $t$ satisfying $-R^2 <s<t<0$ that
\[
\bigl(\vec\beta(t)-\vec \beta(s)\bigr)\cdot \vec \Lambda \leq N R^{-d-1}\abs{\vec \Lambda} \int_{\Omega_{-}[R]}\abs{D\vec u}\,dX+N R^{-d} \abs{\vec \Lambda} \int_{\Omega_{-}[R]} \abs{\vec f}\,dX.
\]
By taking $\vec \Lambda=\vec\beta(t)-\vec \beta(s)$ in the above inequality and using H\"older's inequality and Cauchy's inequality with $\epsilon$, we obtain \eqref{eq8.8ap}.
The proof is complete.
\end{proof}

\begin{lemma} \label{lem:G-06}
Let $\Omega$ be a time-varying $H_1$ (graph) domain in $\bR^{d+1}$.
Let $a^{\alpha\beta}$ satisfy \eqref{eqP-07} and let $\mathscr{E}$ be as in \eqref{eqP-08w}, where $A^{\alpha\beta}_{ij}$ are the coefficients of the operator $\sL$.
Then, there exists $\mathscr{E}_0= \mathscr{E}_0(d, \nu_0,M)>0$ such that if $\mathscr{E} <\mathscr{E}_0$, then the condition $\LH$ is satisfied with $\mu_0=\mu_0(d, \nu_0, M)$, $R_{max}=R_a$, and $N_1=N_1(d,m,\nu_0,M)$. Here, we set $R_a=\infty$ if $\Omega$ is a time-varying $H_1$ graph domain.
\end{lemma}

\begin{proof}
We shall prove below that there exists a number $\mathscr{E}_1=\mathscr{E}_1(d,\nu_0,M)>0$ such that if $\mathscr{E}<\mathscr{E}_1$, then the following holds:
There exist positive constants $\mu_1=\mu_1(d,\nu_0,M)$ and $C_1=C_1(d,m,\nu_0,M)$ such that for any $\tilde{X}\in \partial\Omega=\sP\Omega$ and $R\in (0,R_a)$, if $\vec u$ is a weak solution of $\sL \vec u=0$ in $\Omega_{-}[\tilde{X},R]$ vanishing on $\sP\Omega_{-}[\tilde{X},R]$, then we have
\begin{equation}							\label{eqG-65}
\int_{\Omega_{-}[\tilde{X}, \rho]}\abs{D \vec u}^2 \le C_1 \left(\frac{\rho}{r}\right)^{d+2\mu_1} \int_{\Omega_{-}[\tilde{X},r]} \abs{D \vec u}^2,\quad \forall 0<\rho<r \leq R.
\end{equation}

We also note that by \cite[Lemma~2.2]{CDK}, there is $\mathscr{E}_2=\mathscr{E}_2(d,\nu_0)>0$ such that if $\mathscr{E}<\mathscr{E}_2$, then the following holds:
There exists a constant $\mu_2=\mu_2(d,\nu_0) \in (0,1]$ such that if $\vec u$ is a weak solution of $\sL \vec u=0$ in $Q_{-}(X,R)\subset \Omega$, then we have
\begin{equation}							\label{eqG-65a}
\int_{Q_{-}(X,\rho)} \abs{D\vec u}^2\leq C_2\left(\frac{\rho}{r}\right)^{d+2\mu_2}\int_{Q_{-}(X,r)}\abs{D\vec u}^2, \quad \forall 0<\rho<r\leq R,
\end{equation}
where $C_2=C_2(d,m, \nu_0)$.
Then we combine \eqref{eqG-65} and \eqref{eqG-65a}, via a standard method in boundary regularity theory to conclude that if $\mathscr{E}<\mathscr{E}_1 \wedge \mathscr{E}_2=:\mathscr{E}_0$, then for all $X\in \Omega$ and $0<R<R_a$, the following holds:
If  $\vec u$ is a weak solution of $\sL \vec u=0$ in $\Omega_{-}[X,R]$ vanishing on $S\Omega_{-}[X,R]$, then we have
\begin{equation}					\label{eq6.05jh}
\int_{\Omega_{-}[X,\rho]} \abs{D\vec u}^2\leq N \left(\frac{\rho}{r}\right)^{d+2\mu_0}\int_{\Omega_{-}[X,r]}\abs{D\vec u}^2, \quad \forall 0<\rho<r\leq R,
\end{equation}
where $\mu_0=\mu_1\wedge \mu_2$ and $N=N(d,m,\nu_0,M)$.
By Lemma~\ref{lem8.1ap}, the estimate \eqref{eq6.05jh},  and the energy inequality of Brown et. al \cite{BHL}, we have for all $Y\in Q_{-}(X,R/4)$ and $r\in (0,R/4]$ that
\begin{align*}
\int_{Q_{-}(Y,r)} \abs{\tilde{\vec u}-(\tilde{\vec u})_{Y,r}}^2 &\leq N r^2 \int_{\Omega_{-}[Y,r]} \abs{D \vec u}^2\leq N r^2 \left(\frac{r}{R}\right)^{d+2\mu_0}\int_{\Omega_{-}[Y,R/4]}\abs{D\vec u}^2\\
&\leq N\left(\frac{r}{R}\right)^{d+2+2\mu_0} \int_{\Omega_{-}[Y,R/2]} \abs{\vec u}^2 \leq N r^{d+2+2\mu_0} R^{-2\mu_0} \fint_{Q_{-}(X,R)} \abs{\tilde{\vec u}}^2,
\end{align*}
where $N=N(d,m,\nu_0,M)$.
Then by the Campanato's characterization of H\"older continuous functions (see e.g., \cite[Lemma~2.5]{CDK}), we obtain
\[
[\tilde{\vec u}]_{\mu_0/2,\mu_0;Q_{-}(X,R/4)}^2 \leq C R^{-2\mu_0} \fint_{Q_{-}(X,R)} \abs{\tilde{\vec u}}^2.
\]
Then, the above inequality together with a standard covering argument yields part i) of the condition $\LH$.
The other part of the condition $\LH$ is similarly obtained.

Now, it only remains for us to prove the estimate \eqref{eqG-65}.
For $\tilde{X}\in \partial\Omega$ and $R\in (0,R_a)$ given, let $\vec u$ be a weak solution of $\sL \vec u =0$ in $\Omega_{-}[\tilde{X},R]$ vanishing on $\sP \Omega_{-}[\tilde{X},R]$.
Denote by $\sL_0$ the parabolic operator acting on scalar functions $v$ as follows:
\[
\sL_0 v = v_t - D_\alpha(a^{\alpha\beta}D_\beta v).
\]
For $r \in (0, R]$, let $v^i$ be a unique weak solution in $V_2(\Omega_{-}(\tilde{X},r))$ of the problem
\[
\left\{\begin{array}{rcl}
\sL_0 v^i=0&\text{in}&\Omega_{-}[\tilde{X},r],\\
v^i= u^i&\text {on }&\sP(\Omega_{-}[\tilde{X},r]),
\end{array}\right.
\]
where $i=1,\ldots,m$.
Existence of such $v^i$ follows from Brown et al. \cite{BHL}.
We claim that there are positive constants $\mu=\mu(d,\nu_0,M)$ and $N=N(d,m,\nu_0,M)$ such that the following estimate holds:
\begin{equation}							\label{eqG-66}
\int_{\Omega_{-}[\tilde{X},\rho]}\abs{D\vec v}^2 \leq N\left(\frac{\rho}{r}\right)^{d+2\mu}\int_{\Omega_{-}[\tilde{X},r]}\abs{D\vec v}^2,\quad \forall 0<\rho<r.
\end{equation}

We may assume that $\rho<r/8$ because otherwise \eqref{eqG-66} becomes trivial.
Since each $v^i$ vanishes on $\sP\Omega_{-}[\tilde{X},r]$, it follows from \cite[Theorem~6.32]{Lieberman} and \cite[Theorem~6.30]{Lieberman} that there exist $\mu=\mu(d,\nu_0,M)>0$ and $N=N(d, \nu_0,M)>0$ such that
\begin{equation}							\label{eqG-67}
\osc_{\Omega_{-}[\tilde{X},2\rho]} v^i \leq N \rho^{\mu} r^{-\mu}\sup_{\Omega_{-}[\tilde{X},r/4]} \abs{v^i} \leq N \rho^{\mu} r^{-\mu-d/2-1}\norm{v^i}_{L_2(\Omega_{-}[\tilde{X},r/2])}.
\end{equation}
In particular, the estimate \eqref{eqG-67} implies $v^i(\tilde{X})=0$.
Then, by the energy inequality of Brown et al. \cite{BHL} and \cite[Lemma~4.2]{HK07},  we obtain (recall that $\rho<r/8$)
\begin{align*}
\int_{\Omega_{-}[\tilde{X},\rho]} \abs{D v^i}^2 &\leq N \rho^{-2} \int_{\Omega_{-}[\tilde{X},2\rho]}\abs{v^i}^2= N \rho^{-2}\int_{\Omega_{-}[\tilde{X},2\rho]}\abs{v^i(Y)-v^i(\tilde{X})}^2\,dY \\
&\leq N \rho^{d}\left(\osc_{\Omega_{-}[\tilde{X},2\rho]}\,v^i\right)^2 \leq N \left(\frac{\rho}{r}\right)^{d+2\mu} r^{-2}\int_{\Omega_{-}[\tilde{X},r/2]}\abs{v^i}^2\\
&\leq N\left(\frac{\rho}{r}\right)^{d+2\mu}\int_{\Omega_{-}[\tilde{X},r]}\abs{D v^i}^2,\qquad i=1,\ldots,m.
\end{align*}
where $N=N(d,\nu_0,M)$.
This completes the proof of the estimate \eqref{eqG-66}.
Next, notice that $\vec w:=\vec u-\vec v$ belongs to $V_2(\Omega_{-}[\tilde{X},r])$, vanishes on $\sP(\Omega_{-}[\tilde{X},r])$, and satisfies weakly
\[
\sL_0 \vec w = D_\alpha\left(\bigl(\vec A^{\alpha\beta}-a^{\alpha\beta}I_m \bigr)D_\beta\vec u\right).
\]
Therefore, by the energy inequality of Brown et al. \cite{BHL}, we obtain
\begin{equation}							\label{eqG-70w}
\int_{\Omega_{-}[\tilde{X},r]}\abs{D\vec w}^2\leq N \mathscr{E}^2 \int_{\Omega_{-}[\tilde{X},r]}\abs{D\vec u}^2,
\end{equation}
where $\mathscr{E}$ is defined as in \eqref{eqP-08w}.
By combining \eqref{eqG-66} and \eqref{eqG-70w}, we obtain
\[
\int_{\Omega_{-}[\tilde{X}, \rho]}\abs{D\vec u}^2 \leq N\left(\frac{\rho}{r}\right)^{d+2\mu} \int_{\Omega_{-}[\tilde{X},r]}\abs{D\vec u}^2+N\mathscr{E}^2 \int_{\Omega_{-}[\tilde{X},r]} \abs{D\vec u}^2, \quad \forall 0<\rho<r.
\]
Now, choose a number $\mu_1\in (0,\mu)$.
Then, by a well known iteration argument (see, e.g., \cite[\S III.2]{Gi83}), we find that there exists $\mathscr{E}_1$ such that if $\mathscr{E}< \mathscr{E}_1$, then we have the estimate \eqref{eqG-65}. 
The lemma is proved.
\end{proof}

\begin{acknowledgment}
We thank Steve Hofmann for valuable comments.
Hongjie Dong was partially supported by the National Science Foundation under agreement No.~DMS-0800129.
Seick Kim was supported by supported by Mid-career Researcher Program through NRF grant funded by the MEST (No.~2010-0027491) and also WCU(World Class University) program through the National Research Foundation of Korea(NRF) funded by the Ministry of Education, Science and Technology  (R31-2008-000-10049-0).
\end{acknowledgment}


\end{document}